\documentclass[11pt]{article}
\usepackage{}
\usepackage{enumerate}
\usepackage{mathbbold}
\usepackage{bbm}
\usepackage{amsfonts}
\usepackage[T1]{fontenc}
\usepackage{latexsym,amssymb,amsmath,amsfonts,amsthm}
\usepackage{graphicx, color}
\usepackage{subfigure}
\usepackage{epsf}
\usepackage{epstopdf}
\usepackage{amssymb}
\usepackage{mathrsfs}
\usepackage[ruled,linesnumbered]{algorithm2e}
\usepackage{floatrow}
\usepackage[pagewise]{lineno}

\newcommand{\R}{{\mathbb R}}

\newcommand{\C}{{\mathbb C}}

\newcommand{\be}{\begin{eqnarray}}
\newcommand{\ben}{\begin{eqnarray*}}
\newcommand{\en}{\end{eqnarray}}
\newcommand{\enn}{\end{eqnarray*}}
\newcommand{\ba}{\backslash}
\newcommand{\pa}{\partial}

\newcommand{\ov}{\overline}

\newcommand{\eps}{\epsilon}

\newcommand{\Om}{\Omega}

\newcommand{\wi}{\widetilde}

\newcommand{\hx}{\hat{x}}
\newtheorem{theorem}{Theorem}[section]
\newtheorem{lemma}[theorem]{Lemma}

\newtheorem{definition}[theorem]{Definition}
\newtheorem{remark}[theorem]{Remark}

\newfloatcommand{figurebox}{figure}[\nocapbeside][\dimexpr(\textwidth-\columnsep)/2\relax]
\newfloatcommand{tablebox}{table}[\nocapbeside][\dimexpr(\textwidth-\columnsep)/2\relax]

\begin{document}
\title{Generalized Rellich's lemmas, uniqueness theorem and inside-out duality for scattering poles}
\author{Xiaodong Liu\thanks{Academy of Mathematics and Systems Science,
Chinese Academy of Sciences, Beijing 100190, China. Email: xdliu@amt.ac.cn}
,\and
Jiguang Sun\thanks{Department of Mathematical Sciences, Michigan Technological University, Houghton, MI 49931, USA. Email: jiguangs@mtu.edu}
,\and
Lei Zhang\thanks{Corresponding Author, Department of Mathematics, Zhejiang University of Technology, Hangzhou 310014, China. Email: zhanglei@zjut.edu.cn}}
\date{}

\maketitle
\begin{abstract}
Scattering poles correspond to non-trivial scattered fields in the absence of incident waves and play a crucial role in the study of wave phenomena. These poles are complex wavenumbers with negative imaginary parts. In this paper, we prove two generalized Rellich's lemmas for scattered fields associated with complex wavenumbers. These lemmas are then used to establish uniqueness results for inverse scattering problems. We further explore the inside-out duality, which characterizes scattering poles through the linear sampling method applied to interior scattering problems. Notably, we demonstrate that exterior Dirichlet/Neumann poles can be identified without prior knowledge of the actual sound-soft or sound-hard obstacles. Numerical examples are provided to validate the theoretical results.
\end{abstract}

\section{Introduction}
Scattering resonances play a crucial role in various applications, with extensive research contributing a wealth of significant insights and discoveries (see, e.g., \cite{Dolph, DyatlovZworski2019, LaxPhillips1989, KellerRubinowGoldstein1963, Melrose, SjostrZworski1991, Taylor1986}). The wavenumber-dependent solution operator for the exterior scattering problem is defined on the upper complex plane and possesses a meromorphic continuation, exhibiting poles in the lower half of the complex plane. These poles, known as scattering poles or resonances, encapsulate crucial physical information about the scattering object. This motivates us to investigate the scattering problems when the wavenumber has a negative imaginary part.

In the first part of this paper, we define the far-field pattern of the scattered field for wavenumbers with negative imaginary parts and prove two generalized Rellich's lemmas. We then apply the results to demonstrate that the scattered field vanishes whenever the corresponding far-field pattern does. Additionally, we establish the uniqueness of a sound-soft obstacle under the condition that the far-field patterns are identical for non-real wavenumbers. These results provide valuable insights into scattering phenomena and related inverse problems \cite{Labreuche1998}.

In the context of scattering theory, the inside-out duality, which explores the relationship between interior eigenvalues and scattering operators, has been the subject of investigation over the past decade \cite{CCC2008, CCH2010, KL2013, LS2014,S2011}. Consequently, interior eigenvalues, including Dirichlet, Neumann and transmission eigenvalues, can be approximated from scattering data. This is typically achieved by examining the injectivity of the associated scattering operator.

The inside-out duality between scattering poles and the interior scattering problem was recently investigated in \cite{CCH2020}, where a conceptually unified approach for characterizing both scattering poles and interior eigenvalues for a given scattering problem was proposed. Motivated by \cite{CCH2020}, we conducted a detailed study of this inside-out duality for scattering poles of sound-soft/hard obstacles. Our findings demonstrate that the linear sampling method for interior scattering problems can effectively characterize scattering poles. Numerical examples indicate that the characterization of small poles is quite accurate for smooth obstacles, though less satisfactory for obstacles with corners.

While preparing the manuscript, we came across a relevant paper by Cakoni et al. \cite{CHZ2024}. Their focus is on developing an algorithm for computing scattering poles, building on the duality explored in \cite{CCH2020}. The numerical implementation differs from what we will present. Both papers contribute to the study of the inside-out duality for scattering poles, a topic that warrants further investigation.

The rest of the paper is organized as follows. In Section 2, we present the scattering problem, the outgoing condition, and the asymptotic behavior of the Hankel function. We then prove two generalized Rellich's lemmas, which are used to show that a vanishing far-field pattern implies a vanishing scattered field and uniqueness of the inverse problem. In Section 3, we provide inside-out duality for the scattering poles of sound-soft and sound-hard obstacles. Section 4 contains numerical examples that characterize the scattering poles using the linear sampling method for interior scattering problems. Finally, in Section 5, we draw conclusions and discuss potential future work.

\section{Scattering Problems and Generalized Rellich's Lemmas}
Consider the scattering of time harmonic acoustic waves by a sound-soft obstacle in $\R^2$. Let $D\subset\R^2$ be an open and bounded domain with
Lipschitz-boundary $\pa D$ such that the exterior $\R^2\ba\ov{D}$ is connected. Denote by $\R_{\leq 0}$ the set of non-positive real numbers. 
Given an incident field $u^i$ which is solution of the
Helmholtz equation $\Delta u^i +k^2 u^i = 0$ in $\R^2$ (except for possibly a subset of measure zero in the exterior of $D$), 
the obstacle $D$ gives rise to a scattered field $u^s \in H_{loc}^1(\R^2\ba\ov{D})$ satisfying the Helmholtz equation
\be
\label{Helmholtz}\Delta u^s+k^2u^s&=&0 \quad\text{in } \R^2\ba\ov{D},
\en
where $k\in\C\ba\R_{\leq 0}$ is the wavenumber.
For a sound-soft obstacle, the scattered field satisfies the Dirichlet boundary condition
\be
\label{SoundSoft} u^s = f\quad\text{on } \partial D
\en
with $f=-u^i|_{\pa D}\in H^{1/2}(\pa D)$. Similarly, the scattering from a sound-hard obstacle leads to the Neumann boundary condition
\be
\label{SoundHard}\frac{\partial u^s}{\partial \nu}= f \quad\text{on } \partial D
\en
with $f=-\frac{\pa u^i}{\pa\nu}\Big|_{\pa D}\in H^{-1/2}(\pa D)$. Here, $\nu$ is the unit outward normal to $\pa D$.

Let $H^{(1,2)}_{n}$ denote the Hankel functions of the first and second kind of order $n$.
The scattered field $u^s$ is {\em outgoing} if $u^s$ has the following Green's representation
\be\label{outgoing}
u^s(x)=\int_{\pa D}\left\{\frac{\pa \Phi_k(x,y)}{\pa\nu(y)}u^s(y)-\Phi_k(x,y)\frac{\pa u^s}{\pa\nu}(y)\right\}ds(y),\quad x\in \R^2\ba\ov{D},
\en
where
\ben
\Phi_k(x,y):=\frac{i}{4}H^{(1)}_{0}(k|x-y|)\,\quad x\in\R^2\ba\{y\}
\enn
is the fundamental solution of the Helmholtz equation in $\R^2$.
The  {\em outgoing} wave $u^s$ is analytic in $\R^2\ba\ov{D}$ by the analyticity of the fundamental solution for $x\neq y$.
If $\Im k\geq 0$, the {\em outgoing} condition \eqref{outgoing} is equivalent to the well known Sommerfeld radiation condition 
\ben
\label{Sommerfeld} \lim_{r\to\infty}\sqrt r\left(\frac{\partial u^s}{\partial r}-iku^s\right)=0,\quad r:=|x|.
\enn
This has been shown in Theorems 3.2-3.3 of \cite{CK-integral} for the three-dimensional case. The two-dimensional case can be proved similarly.
However, such an equivalence property does not hold for the case of $\Im k < 0$. Actually, the asymptotic behaviour given in \eqref{us-asy} shows that the {\em outgoing} solution $u^s$ with $\Im k <0$ exponential grows as $|x|\rightarrow \infty$. 

It is well known that, for $k\in {\mathbb C}_{+}:=\{k\in\C: \Im k\geq 0\}$, there exists a unique solution to the exterior scattering problem \eqref{Helmholtz}+\eqref{SoundSoft} + \eqref{outgoing} (or \eqref{Helmholtz}+ \eqref{SoundHard} + \eqref{outgoing}). 
Denote the scattering operator by $\mathcal {B}(k)$ such that
\be\label{SoluitonOperator}
u^s = \mathcal {B}(k)f.
\en
The operator function $\mathcal {B}(k)$ has a meromorphic continuation to the lower-half complex plane with discrete scattering poles in ${\mathbb C}_{-}:=\{k\in\C: \Im k <0\}$.
\begin{definition}
We call $k\in {\mathbb C}_{-}$ a Dirichlet (or Neumann) pole if the corresponding exterior homogeneous Dirichlet boundary value problem \eqref{Helmholtz}+\eqref{SoundSoft} + \eqref{outgoing} (or \eqref{Helmholtz}+ \eqref{SoundHard} + \eqref{outgoing}) with $f=0$ admits a nontrivial radiating solution $u^s$.
\end{definition}
The scattering poles correspond to non-zero scattered fields with zero incident fields.
There are rich results on the scattering poles.
We refer the readers to Chapter 9 of \cite{Taylor1986} and the recent monograph \cite{DyatlovZworski2019}. Numerical methods for computing scattering poles can be found in \cite{MaSun2023, XiGongSun2024} and the references therein.

With the help of the formula $H^{(1)\prime}_{0}(z)=-H^{(1)}_{1}(z)$ and the asymptotic behavior of the Hankel functions for large arguments \cite{AS}\be\label{Hankel-asy}
H^{(1,2)}_{n}(z)=\sqrt{\frac{2}{\pi z}}e^{\pm i(z-\frac{n\pi}{2}-\frac{\pi}{4})}\left\{1+O\Big(\frac{1}{|z|}\Big)\right\},\quad |z|\rightarrow\infty,\, \arg{z}<\pi,
\en
the {\em outgoing} solution $u^s$ has the following asymptotic behaviour 
\be\label{us-asy}
u^s(x)=\frac{e^{ik|x|}}{\sqrt{|x|}}\left\{u^{\infty}(\hx)+O\Big(\frac{1}{|x|}\Big)\right\}
\en
uniformly in all directions $\hx=x/|x|$, where the function $u^\infty$ defined on the unit circle $\mathbb{S}^1$ is known as the far-field pattern of $u^s$.

\subsection{Generalized Rellich's Lemmas}

Rellich's lemma states that, for $k \in {\mathbb R}_+:=\{k \in \mathbb R: k > 0\}$, if $u\in C^2(\R^2\ba\ov{D})$ is a solution to the Helmholtz equation satisfying 
\be\label{Rellich}
\lim_{r\rightarrow\infty} \int_{|x|=r}|u(x)|^2ds=0,
\en 
then $u=0$ in $\R^2\ba\ov{D}$. An important consequence of Rellich's lemma is that the {\em outgoing} wave $u^s$ and its far field pattern $u^{\infty}$ has a one-to-one correspondence. However, Rellich's lemma does not hold for $k$ with $\Im k\neq 0$. In fact, if $\Im k<0$, the {\em outgoing} solution grows exponentially and, as a consequence, \eqref{Rellich} does not hold.  We will show that the one-to-one correspondence between the {\em outgoing} wave $u^s$ and its far field pattern $u^{\infty}$  does hold for all $k\in\C\ba\R_{\leq 0}$. To do so, we need to generalize Rellich's lemma from $k \in \mathbb R_+$ to $k\in\C\ba\R_{\leq 0}$.

\begin{lemma}\label{Rellich-generalized} 
(Generalized Rellich’s Lemma I) Let $k\in\C\ba\R_{\leq 0}$ and let $u\in C^2(\R^2\ba\ov{D})$ be a solution to the Helmholtz equation satisfying 
\be\label{Rellich-assumption}
\lim_{r\rightarrow\infty} e^{2|\Im(kr)|}\int_{|x|=r}|u(x)|^2ds=0,
\en 
then $u=0$ in $\R^2\ba\ov{D}$.
\end{lemma}
\begin{remark}
The case of $k\in \mathbb R_+$ is the classical Rellich's lemma.
\end{remark}
\begin{proof}
Let $(r,\phi)$ denote the polar coordinates. Recall the spherical harmonics 
\ben
Y_n(\hx):=\sqrt{\frac{1}{2\pi}}e^{in\phi}, n=0,\pm 1, \pm 2, \cdots,
\enn
which form a complete orthonormal system in $L^2(\mathbb{S}^1)$. For sufficiently large fixed $|x|=r$, we can expand $u$ as a uniformly convergent series of spherical harmonics
\be\label{anI}
u(x)=\sum_{n=-\infty}^{\infty}a_n(r)Y_n(\hx),
\en
where the coefficients are given by 
\ben
a_n(r)=\int_{\mathbb{S}^1}u(r\hx)\ov{Y_n}(\hx)ds(\hx).
\enn

Due to the orthonormality of the spherical harmonics, it holds that 
\ben
\int_{|x|=r}|u(x)|^2ds=r\sum_{n=-\infty}^{\infty}|a_n(r)|^2.
\enn
From the assumption \eqref{Rellich-assumption}, we find that
\be\label{an-limit}
\lim_{r\rightarrow\infty} re^{2|\Im(kr)|}|a_n(r)|^2 = 0
\en
for all $n=0,\pm 1,\pm 2,\cdots$. Since $u\in C^2(\R^2\ba\ov{D})$, we can differentiate \eqref{anI}, make use of $\Delta u+k^2u=0$ and the orthogonality of $Y_n$ to conclude that $a_n$ are solutions to the Bessel equation
\ben
\frac{d^2 a_n}{dr^2} + \frac{1}{r}\frac{d a_n}{dr} +\Big(k^2-\frac{n^2}{r^2}\Big)a_n =0,
\enn
that is,
\be\label{an}
a_n(r)= \alpha_n H^{(1)}_n(kr) + \beta_n H^{(2)}_n(kr),
\en
where $\alpha_n$ and $\beta_n$ are constants. 
Finally, inserting \eqref{an} into \eqref{an-limit} and using the asymptotic formula \eqref{Hankel-asy}, we obtain
\be\label{lim-final}
\lim_{r\rightarrow \infty} e^{2|\Im(kr)|}\Big[e^{-2\Im(kr)}|\alpha_n|^2+2\Re[\alpha_n\ov{\beta_n}e^{2i(\Re(kr)-\frac{n\pi}{2}-\frac{\pi}{4})}]+e^{2\Im(kr)}|\beta_n|^2\Big]=0.
\en
We divide $k$ into three  cases in terms of the sign of $\Im(k)$.
\begin{itemize}
  \item $\Im(k)=0$.
  
        This is the classical Rellich's lemma with $k>0$. The limit \eqref{lim-final} exists only if the product $\alpha_n\ov{\beta_n}=0$. Inserting this into the limit \eqref{lim-final} we deduce that $|\alpha_n|^2+|\beta_n|^2=0$, which implies that $\alpha_n=\beta_n=0$.
  \item $\Im(k)>0$.

        In this case, the third term of \eqref{lim-final} grows exponentially if $\beta_n\neq 0$. Thus the coefficient $\beta_n$ must vanish. Inserting $\beta_n=0$ into \eqref{lim-final} we further deduce that $\alpha_n=0$.
  \item $\Im(k)<0$.

        In this case,  the first term of \eqref{lim-final} grows exponentially if $\alpha_n\neq 0$. Similar to the second case, we have
        $\alpha_n=\beta_n=0$.
\end{itemize}
Thus, for all $k\in\C\ba\R_{\leq 0}$, $\alpha_n=\beta_n=0$,  $n=0,\pm 1,\pm 2,\cdots$. Therefore $u=0$ outside a sufficiently large disk and $u=0$ in $\R^2\ba\ov{D}$ by analyticity.
\end{proof}

If it is assumed further that the solution $u^s$ to the Helmholtz equation is {\em outgoing}, one has the following generalized Rellich's lemma, which can be used to establish the one-to-one  correspondence between the {\em outgoing} wave $u^s$ and its far field pattern $u^{\infty}$ for all $k\in\C\ba\R_{\leq 0}$.

\begin{lemma}\label{Rellich-generalized-outgoing}
(Generalized Rellich’s Lemma II) Let $k\in\C\ba\R_{\leq 0}$ and let $u\in C^2(\R^2\ba\ov{D})$ be an {\em outgoing} solution to the Helmholtz equation satisfying 
\be\label{Rellich-assumption-outgoing}
\lim_{r\rightarrow\infty} e^{2\Im(kr)}\int_{|x|=r}|u(x)|^2ds=0,
\en 
then $u=0$ in $\R^2\ba\ov{D}$.
\end{lemma}
\begin{proof}
The proof is a slight modification of the proof for Lemma \ref{Rellich-generalized}.
Eqn. \eqref{Rellich-assumption-outgoing} implies that
\be\label{an-limit-outgoing}
\lim_{r\rightarrow\infty} re^{2\Im(kr)}|a_n(r)|^2 = 0 \quad \text{for } n=0,\pm 1,\pm 2,\cdots
\en 
Since $u\in C^2(\R^2\ba\ov{D})$ is an {\em outgoing} solution to the Helmholtz equation, \eqref{an} becomes
\be\label{alphanHn1}
a_n(r)= \alpha_n H^{(1)}_n(kr).
\en
Inserting \eqref{alphanHn1} into \eqref{an-limit-outgoing} and using the asymptotic formula \eqref{Hankel-asy}, it holds that
\ben\label{lim-final-outgoing}
\lim_{r\rightarrow \infty} |\alpha_n|^2=0,
\enn
which implies $\alpha_n=0$ since $\alpha_n$ is independent of $r$. Hence $a_n=0$ for all $n=0,\pm 1,\pm 2,\cdots$ and, consequently, $u=0$ outside a sufficiently large disk. By analyticity, $u=0$ in $\R^2\ba\ov{D}$.
\end{proof}

\begin{theorem}\label{1us-1uinf}
    Let $u\in C^2(\R^2\ba\ov{D})$ be an {\em outgoing} solution to the Helmholtz equation for which the far field pattern $u^{\infty}$ vanishes on an open part of $\mathbb{S}^1$. Then $u=0$ in $\R^2\ba\ov{D}$.
\end{theorem}
\begin{proof}
    By analyticity of the far field pattern with respect to the observation directions, we have $u^{\infty}=0$ on $\mathbb{S}^1$.
    From the expansion \eqref{us-asy} we have
    \ben
\lim_{r\rightarrow\infty} e^{2\Im(kr)}\int_{|x|=r}|u^s(x)|^2ds=0,
\enn 
then $u^s=0$ in $\R^2\ba\ov{D}$ by Lemma \ref{Rellich-generalized-outgoing}.
\end{proof}

We would like to remark that Lemmas \ref{Rellich-generalized} and \ref{Rellich-generalized-outgoing} and Theorem \ref{1us-1uinf} remain valid in three dimensions. 

\subsection{Uniqueness Theorems}
As an application of Theorem \ref{1us-1uinf}, we show that a single far field pattern for $k\in \C\ba\R$ is enough to reconstruct the sound-soft obstacle. Note that the uniqueness result for $k \in \mathbb R_+$ is still a long-standing open problem.

\begin{theorem}
  Let $D_1$ and $D_2$ be two sound-soft obstacles. Then $D_1=D_2$ if one of the following two assumptions holds.
\begin{itemize}
  \item[(A)] Let $k\in \C_{-}$ be a Dirichlet pole, $u^i=0$ and $u^\infty\neq0$. $D_1$ and $D_2$ have the same pole $k$ and such that $u^\infty$ is the far field pattern of two eigenfunctions associated to $k$.
  \item[(B)] For any fixed $k\in \C\ba\R$. The far field patterns coincide for a single incident field. Here, the incident field $u^i$ can be a plane wave $u^i=e^{ikx\cdot d}$ with fixed direction $d\in \mathbb{S}^1$ or a point source $u^i=\Phi_k(x,y)$ with fixed $y \in \mathbb R^2 \setminus \overline{D_{1}\cup D_{2}}$. 
\end{itemize}  
\end{theorem}

\begin{remark}
    A uniqueness result when $k$ is a Dirichlet pole was proposed by Labreuche in \cite{Labreuche1998}. However, the proof uses the classical Rellich's lemma for $k \in \mathbb R_+$, which does not hold if $k$ is a scattering pole. We remove this gap by using the generalized Rellich's lemma \ref{Rellich-generalized-outgoing} and extend the uniqueness result for all $k\in \C\ba\R$.  Our proof follows the arguments given by Labreuche in \cite{Labreuche1998}, which is originated from Schiffer \cite{ColtonKress2019,LaxPhillips1989}.
\end{remark}

\begin{proof}
 Assume that $D_{1}\neq D_{2}$. Since by Theorem \ref{1us-1uinf} the far field pattern uniquely determines the scattered field, the scattered fields $u^{s}_{1}$ and $u^{s}_{2}$ for $D_{1}$ and $D_{2}$ coincide in the unbounded component $G$ of the complement of $D_{1}\cup D_{2}$. Without loss of generality, we can assume that $D^{\ast}:= (\R^{2}\backslash G)\backslash\ov{D_{2}}$ is nonempty. Then $u^{s}_{2}$ is defined
in $D^{\ast}$, and the total field $u_2 = u^{i}+ u^{s}_{2}$ satisfies the Helmholtz equation in $D^{\ast}$ and the homogeneous boundary condition $u_2 = 0$ on $\pa D^{\ast}$. From Green's theorem we have that
\ben
\int_{D^{\ast}} (|\nabla u_2|^2-k^2|u_2|^2)dx=0.
\enn
Taking the imaginary part of the above equation and using the assumption $k\in \C\ba\R$, we obtain that $u_2=0$ in $D^{\ast}$. Furthermore, by the analytic continuation, $u_2=0$ in $\R^2\ba\ov{D_2\cup \{y\}}$.

Under the assumption (A), we have $u^{\infty}\neq 0$. This contradicts to the fact that $u^\infty$ is the far field pattern of $u_2=0$.  

Under the assumption (B), we have $u_2^s=-u^i$ in $\R^2\ba\ov{D_2\cup \{y\}}$. For the plane wave case, it contradicts the fact that the scattered field $u^s$ grows (or decays) exponentially as $|x|\rightarrow 0$ while the plane wave $|u^i(x,d)|=1$ if $x\cdot d=0$. For the point source case, it contradicts the fact that the scattered field $u^s$ is smooth in $\R^2\ba\ov{D_2}$ while the point source $u^i$ is singular at $y$.
\end{proof}

Unfortunately, the domain $D^{\ast}$ may have some cusps and thus Green's theorem may not hold if $u_2\neq 0$ on $\pa D^{\ast}$. Therefore, such an uniqueness result can not be proved as above for other boundary conditions.

\section{Inside-out Duality for Scattering Poles}
In this section, we derive the inside-out duality for scattering poles of sound-soft obstacles and sound-hard obstacles. The duality characterizes the scattering poles using the linear sampling method for the interior scattering problems. For sound-soft obstacles, we present a slightly different characterization of the scattering poles compared to that in \cite{CCH2020}, while using the same framework. The sound-hard case is analyzed in detail, complementing the results of \cite{CCH2020}. To this end, we recall the following interior boundary value problems for Helmholtz equations.

{\bf Interior Dirichlet boundary value problem:}
Given $f\in H^{1/2}(\pa D)$, find $u \in H^1(D)$ such that
\be
\label{IDbvp}\Delta u+k^2u = 0 \quad\text{in } D,\qquad u = f\quad\text{on } \partial D.
\en

{\bf Interior Neumann boundary value problem:}
Given $f\in H^{-1/2}(\pa D)$, find $u \in H^1(D)$ such that
\be
\label{INbvp}\Delta u+k^2u = 0 \quad\text{in } D,\qquad \frac{\pa u}{\pa\nu} = f\quad\text{on } \partial D.
\en

The following lemma on uniqueness is well-known.
\begin{lemma}\label{Uni-Interior}
$k\in {\mathbb C}\ba\R$ is NOT a Dirichlet/Neumann eigenvalue of $-\Delta$ in $D$, i.e., the interior boundary value problem \eqref{IDbvp}/\eqref{INbvp} with $f=0$ admits only trivial solution $u=0$ in $D$.
\end{lemma}




For $k\in {\mathbb C}_{-}$, the existence of the solution to the interior boundary value problems \eqref{IDbvp}/\eqref{INbvp} can be easily established by the standard variational method or the integral equation method. Denote by $\Om$ a bounded domain with Lipschitz boundary $\pa\Om$ such that $\ov{\Om}\subset D$.
We define an operator $G$ by
\be\label{G-Def}
Gf=u|_{\pa\Om},
\en
where $u$ is the solution of the boundary value problem \eqref{IDbvp}/\eqref{INbvp}. The interior regularity property implies that the operator $G$ is compact.

Considering the scattering problem due to a point source $\Phi_k(\cdot, y)$ located at $y\in\pa \Om$, let $u^s_k(\cdot, y)$ be the unique solution of the interior boundary value problems \eqref{IDbvp} (or \eqref{INbvp}) with the Dirichlet boundary value $-\Phi_k(\cdot, y)$ (or Neumann boundary value $-\pa \Phi_k(\cdot, y)/\pa\nu$) on $\pa D$.
The following reciprocity relation holds 
\be\label{Reciprocity Relation}
u^s_k(x, y)= u^s_k(y,x),\quad x,y\in D.
\en
Define the near field operator $\mathcal{N}_k: L^2(\pa \Om) \to L^2(\pa \Om)$ such that
\begin{equation}\label{NFoperator}
\Big(\mathcal{N}_k g\Big)(x) = \int_{\pa \Om}u^{s}_k(x,y)g(y)ds(y),\quad x\in \pa\Om.
\end{equation}
Introduce the single layer potential operator $\mathcal {S}_k: L^2(\pa \Om) \to H^1_{loc}(\R^2\ba{\ov{D}})$ such that
\ben
\Big(\mathcal {S}_k g\Big)(x):=\int_{\pa\Om} \Phi_k(x,y)g(y)ds(y),\quad x\in \R^2\ba{\ov{D}}.
\enn

\begin{theorem} \label{S-Properties}
If $k\in {\mathbb C}_{-}$, then the operator $\mathcal {S}_k: L^2(\pa \Om) \to H^1_{loc}(\R^2\ba{\ov{D}})$ is injective. If $k\in {\mathbb C}_{-}$ is not a scattering pole, then $\mathcal {S}_k: L^2(\pa \Om) \to H^1_{loc}(\R^2\ba{\ov{D}})$ has a dense range.
\end{theorem}
\begin{proof}
Given $g\in L^2(\pa \Om)$, define
     \ben
     u_g (x) := \int_{\pa\Om} \Phi_k(x,y)g(y)ds(y),\quad x\in \R^2\ba{\pa \Om}.
     \enn
     Then $u_g$ is a solution to the Helmholtz equation in $\Om$ and in $\R^2\ba\ov{\Om}$ and satisfies the outgoing condition \eqref{outgoing}.
     Let $u_g=\mathcal {S}_k g=0$ in $\R^2\ba\ov{D}$, by the unique continuation principle, we have $u_g=0$ in $\R^2\ba\ov{\Om}$. Furthermore, by the continuation of the single layer potential, we have $u_g=0$ on $\pa \Om$. By Lemma \ref{Uni-Interior}, we have $u_g=0$ in $\Om$. Now the jump relation of the single layer potential implies $g=0$. Hence, the operator $\mathcal {S}_k: L^2(\pa \Om) \to H^1_{loc}(\R^2\ba{\ov{D}})$ is injective.

Next we assume that $k\in \C_{-}$ is not a Dirichlet pole. We introduce an auxiliary operator $A_k: L^2(\pa \Om) \to H^{1/2}(\pa D)$ defined by
       \ben
        \Big(A_k g\Big)(x):=\int_{\pa\Om} \Phi_k(x,y)g(y)ds(y),\quad x\in \pa D.
       \enn
     One has the factorization $\mathcal {S}_k = \mathcal {B}(k)A_k$,  where $\mathcal {B}(k)$ is the scattering operator defined in \eqref{SoluitonOperator}. Note that $\mathcal {B}(k)$ is well-defined since $k\in {\mathbb C}_{-}$ is not a scattering pole. The denseness of the range of $\mathcal {S}_k$ then follows from the fact that the auxiliary operator $A_k$ has a dense range in $H^{1/2}(\pa D)$. Interchanging the order of integration we observe that the adjoint operator $A^\ast_k: H^{-1/2}(\pa D) \to L^2(\pa \Om)$ of $A_k$ is
     defined by
     \ben
     \Big(A^\ast_k h\Big)(y):=\int_{\pa D} \ov{\Phi_k(x,y)}h(x)ds(x),\quad y\in \pa\Om.
     \enn
     Define the single layer potential
     \ben
     v_h(y):=\int_{\pa D}\Phi_k(x,y)\ov{h(x)}ds(x), \quad y\in \R^2\ba\pa D.
     \enn
     Thus $v_h$ is a solution to the Helmholtz equation in both $D$ and $\R^2\ba\ov{D}$ and satisfies the outgoing condition \eqref{outgoing}.
     Letting $A^\ast_k h=0$ on $\pa\Om$, we have $v_h=0$ on $\pa\Om$. Note that $v_h$ is a solution to the Helmholtz equation in $D$,  Lemma \ref{Uni-Interior} yields $v_h=0$ in $\Om$. Then $v_h=0$ in $D$ by analytic continuation. In particular, $v_h=0$ on $\pa D$ in the trace sense. Noting that $v_h$ is also an outgoing solution to the Helmholtz equation in $\R^2\ba\ov{D}$, we have, by the assumption $k\in {\mathbb C}_{-}$ is not a Dirichlet pole, that $v_h=0$ in $\R^2\ba\ov{D}$. Thus, $h=0$ by the jump relation of the single layer potential. That is, the adjoint operator $A^\ast_k$ is injective and therefore the auxiliary operator $A_k$ has a dense range in $H^{1/2}(\pa D)$.
     
     Finally, we consider the case that $k\in \C_{-}$ is not a Neumann pole. To do so, we introduce another auxiliary operator $B_k: L^2(\pa \Om) \to H^{1/2}(\pa D)$ defined by
       \ben
        \Big(B_k g\Big)(x):=\frac{\pa}{\pa\nu(x)}\int_{\pa\Om} \Phi_k(x,y)g(y)ds(y),\quad x\in \pa D.
       \enn
     Then, we have the factorization $\mathcal {S}_k = \mathcal {B}(k)B_k$. Using the fact that the exterior Neumann problem is well posed if $k$ is not a Neumann pole, the denseness of the range of $\mathcal {S}_k$ follows from the fact that the auxiliary operator $B_k$ has a dense range in $H^{-1/2}(\pa D)$. Interchanging the order of integration, the adjoint operator $B^\ast_k: H^{1/2}(\pa D) \to L^2(\pa \Om)$ of $B_k$ is
     defined by
     \ben
     \Big(B^\ast_k h\Big)(y):=\int_{\pa D} \ov{\frac{\pa\Phi_k(x,y)}{\pa\nu(x)}}h(x)ds(x),\quad y\in \pa\Om.
     \enn
     Define the double layer potential
     \ben
     w_h(y):=\int_{\pa D}\frac{\pa\Phi_k(x,y)}{\pa\nu(x)}\ov{h(x)}ds(x), \quad y\in \R^2\ba\pa D.
     \enn
     Then $w_h$ is a solution to the Helmholtz equation in both $D$ and $\R^2\ba\ov{D}$ and satisfies the outgoing condition \eqref{outgoing}.
     Letting $B^\ast_k h=0$ on $\pa\Om$, we have $w_h=0$ on $\pa\Om$. Note that $w_h$ is a solution to the Helmholtz equation in $D$, Lemma \ref{Uni-Interior} yields $w_h=0$ in $\Om$. Thus $w_h=0$ in $D$ by analytic continuation. In particular, $\frac{\pa v_h}{\pa\nu}=0$ on $\pa D$ in the trace sense. Noting that $w_h$ is also an outgoing solution to the Helmholtz equation in $\R^2\ba\ov{D}$ and $k\in {\mathbb C}_{-}$ is not a Neumann pole, we have that $w_h=0$ in $\R^2\ba\ov{D}$. By the jump relation of the double layer potential $h=0$. That is, the adjoint operator $B^\ast_k$ is injective and therefore $B_k$ has dense range in $H^{-1/2}(\pa D)$.
\end{proof}

We first study the inside-out duality for the case of Dirichlet boundary condition in Section \ref{subsec-D}. The case of Neumann boundary condition is considered in Section \ref{subsec-N}, which is analogous to the Dirichlet case.

\subsection{Dirichlet Case}\label{subsec-D}
Letting $w\in H^1_{loc}(\R^2\ba{\ov{D}})$ be an outgoing solution to the Helmholtz equation in $\R^2\ba{\ov{D}}$, we define the operator $\mathcal {G}_k: H^1_{loc}(\R^2\ba{\ov{D}}) \to L^2(\pa \Om)$ by $\mathcal {G}_k w=u_w|_{\pa \Om}$, where $u_w\in H^1(D)$ solves
\ben
\Delta u_w+k^2 u_w= 0 \quad\text{in } D,\qquad u_w = w|_{\pa D} \quad\text{on } \pa D.
\enn
Assume that $k\in {\mathbb C}_{-}$. Then $\mathcal{N}_k: L^2(\pa \Om) \to L^2(\pa \Om)$ has a factorization
     \be\label{Factorization}
     \mathcal{N}_k=\mathcal{G}_k\mathcal{S}_k.
     \en
Assuming further that $k\in {\mathbb C}_{-}$ is not a Dirichlet pole, we have $\Phi_k(\cdot, z)\in \mathscr{R}(\mathcal {G}_k)$ if $z\in \R^2\ba\ov{D}$ (see Lemma 2.10 of \cite{CCH2020}). 

To characterize the Dirichlet poles from the scattered fields for the interior cavity scattering problems, we need the following two theorems.

\begin{theorem}\label{LSM-D}
Assume that $k\in {\mathbb C}_{-}$ is not a Dirichlet pole. Let $z\in \R^2\ba\ov{D}$. Then for every $\eps$, $0<\eps<1$, there exists $g_\eps^z\in L^2(\pa\Om)$ such that
\be\label{Ng=Phi-D}
\lim_{\eps\rightarrow 0}\|\mathcal{N}_k g^{\eps}_z - \Phi_k(\cdot, z)\|_{L^2(\pa\Om)}=0
\en
and
\be\label{Sg<M}
\|\mathcal{S}_k g^{\eps}_z\|_{H^1_{loc}(\R^2\ba{\ov{D}})} < M
\en
for some constant $M>0$. 
\end{theorem}
\begin{proof}
Since $k\in {\mathbb C}_{-}$ is not a Dirichlet pole, the exterior Dirichlet problem is well posed. Denote by $w_z\in H^1_{loc}(\R^2\ba{\ov{D}})$ the outgoing solution with boundary value $w_z=\Phi_k(\cdot, z)$ on $\pa D$. Consequently, $\mathcal {G}_k w_z = \Phi_k(\cdot, z)$ on $\pa\Om$. Furthermore, using Theorem \ref{S-Properties}, we deduce that, for every $\eps>0$, there exists $g_\eps^z\in L^2(\pa\Om)$ such that
\be\label{Sg-wz}
\|\mathcal{S}_k g^{\eps}_z-w_z\|_{H^1_{loc}(\R^2\ba{\ov{D}})}<\frac{\eps}{\|\mathcal {G}_k\|}.
\en
The limit \eqref{Ng=Phi-D} follows immediately from the inequality \eqref{Sg-wz} and the factorization \eqref{Factorization}.
By the triangular inequality, the inequality \eqref{Sg-wz} implies
\ben
\|\mathcal{S}_k g^{\eps}_z\|_{H^1_{loc}(\R^2\ba{\ov{D}})}\leq \|w_z\|_{H^1_{loc}(\R^2\ba{\ov{D}})}+\frac{\eps}{\|\mathcal {G}_k\|}.
\enn
Then the inequality \eqref{Sg<M} holds by taking $M=\|w_z\|_{H^1_{loc}(\R^2\ba{\ov{D}})}+1/\|\mathcal {G}_k\|$.
\end{proof}

The existence of $g_\eps^z\in L^2(\pa\Om)$ such that the limit holds has been obtained in Theorem 2.11 in \cite{CCH2020}. We want to recall that the uniform boundedness of $\|\mathcal{S}_k g^{\eps}_z\|_{H^{1/2}(\pa D)}$ has also been proved in Theorem 2.11 in \cite{CCH2020}. As pointed out in \cite{CCH2020}, if $k\in {\mathbb C}_{-}$ is not a Dirichlet pole, the two norms $\|\mathcal{S}_k g^{\eps}_z\|_{H^{1/2}(\pa D)}$ and $\|\mathcal{S}_k g^{\eps}_z\|_{H^1_{loc}(\R^2\ba{\ov{D}})}$ are equivalent. However, this is not true if $k\in {\mathbb C}_{-}$ is a Dirichlet pole. The following Theorem \ref{pole-D} shows that, for $z\in\R^2\ba\ov{D}$, the norm $\|\mathcal{S}_k g^{\eps}_z\|_{H^1_{loc}(\R^2\ba{\ov{D}})}$ cannot be bounded if $k\in {\mathbb C}_{-}$ is a Dirichlet pole.

\begin{theorem}\label{pole-D}
Assume that $k\in {\mathbb C}_{-}$ is a Dirichlet pole and 
\be\label{Ng-Phi}
\lim_{\eps\rightarrow 0}\|\mathcal{N}_k g^{\eps}_z - \Phi_k(\cdot, z)\|_{L^2(\pa\Om)}=0
\en
holds for all $z\in \R^2\ba\ov{D}$.
Then, for almost every $z\in \R^2\ba\ov{D}$, $\|\mathcal{S}_k g^{\eps}_z\|_{H^{1}_{loc}(\R^2\ba\ov{D})}$ can not be bounded as $\eps\rightarrow 0$.
\end{theorem}
\begin{proof}
Assume on the contrary that
\ben
\|\mathcal{S}_k g^{\eps_n}_z\|_{H^{1}_{loc}(\R^2\ba\ov{D})}\leq M
\enn
for some constant $M>0$, some sequence $\eps_n\rightarrow 0$ and all
$z\in D_0\subset \R^2\ba\ov{D}$ where $D_0$ has a positive measure. Hence, by selecting subsequences and relabeling, there exists a weakly convergent sequence  $v_n:=\mathcal{S}_k g^{\eps_n}_z \rightharpoonup v$ for some $v\in H^{1}_{loc}(\R^2\ba\ov{D})$. Denote by $T: H^{1}_{loc}(\R^2\ba\ov{D}) \to H^{1/2}(\pa D)$ the trace operator.
Then $T v_n \rightharpoonup Tv$ on $H^{1/2}(\pa D)$.

Let $\wi{v}\in H^{1}_{loc}(\R^2\ba\ov{D})$ be an outgoing solution to the Helmholtz equation in $\R^2\ba\ov{D}$ with boundary value $\wi{v}=v$ on $\pa D$. This implies $T\wi{v}=Tv$ and therefore $T v_n \rightharpoonup T\wi{v}$ on $H^{1/2}(\pa D)$.
By Green's representation theorem, we have that
\be\label{wiv}
\wi{v}(x)=\int_{\pa D}\left\{\frac{\pa \Phi_k(x,y)}{\pa\nu(y)}\wi{v}(y)-\Phi_k(x,y)\frac{\pa \wi{v}}{\pa\nu}(y)\right\}ds(y),\quad x\in \R^2\ba\ov{D}.
\en
Note that $\mathcal{G}_k = GT$ where $G: H^{1/2}(\pa D) \to L^2(\pa \Om)$ is the compact operator defined by \eqref{G-Def}.
Furthermore, with the help of the factorization $\mathcal{N}_k=\mathcal{G}_k\mathcal{S}_k$, we have that
\be\label{Gvn-Gv}
\mathcal{N}_k g^{\eps}_z = \mathcal{G}_k\mathcal{S}_k g^{\eps}_z = \mathcal{G}_k v_n = GT v_n \rightarrow GT \wi{v} = \mathcal{G}_k \wi{v} \quad\mbox{in}\, L^2(\pa\Om).
\en
Denote by $u_z$ the solution of
\ben
\Delta u_z +k^2 u_z = 0 \quad\mbox{in}\, D, \qquad u_z=\wi{v}|_{\pa D}\quad\mbox{on}\, \pa D.
\enn
Then $\mathcal{G}_k \wi{v} = u_z$ on $L^2(\pa\Om)$. Combination of \eqref{Gvn-Gv} and \eqref{Ng-Phi} yields that $u_z=\Phi(\cdot,z)$ on $\pa\Om$. Using Lemma \ref{Uni-Interior} and the analytic continuation, we have $u_z=\Phi(\cdot, z)$ in $D$ and thus
\be\label{wiv=Phi}
\wi{v}=\Phi(\cdot, z)\quad\mbox{on}\, \pa D.
\en

We recall the single and double layer operator $S$ and $K$ given by
\be
S\phi(x)&:=&\int_{\pa D}\Phi_k(x,y)\phi(y)ds(y), \quad x\in\pa D,\cr
\label{K}K\phi(x)&:=&\int_{\pa D}\frac{\pa\Phi_k(x,y)}{\pa\nu(y)}\phi(y)ds(y), \quad x\in\pa D
\en
and the normal derivative operator $K^{'}$ given by
\ben
K^{'}\phi(x):=\int_{\pa D}\frac{\pa\Phi_k(x,y)}{\pa\nu(x)}\phi(y)ds(y), \quad x\in\pa D.
\enn
By the jump relation, the representation \eqref{wiv} implies
\be\label{wiv-b}
\wi{v}|_{\pa D} = \Big(\frac{1}{2}I+K\Big)\wi{v}|_{\pa D} - S \frac{\pa\wi{v}}{\pa\nu}\Big|_{\pa D}.
\en

Since $k\in {\mathbb C}_{-}$ is a Dirichlet pole, we take $w_0$ be a non-trivial outgoing solution in $\R^2\ba\ov{D}$ with homogeneous boundary data $w_0= 0$ on $\pa D$.
By Green's representation theorem again,
\ben
w_0(x)
&=&\int_{\pa D}\left\{\frac{\pa \Phi_k(x,y)}{\pa\nu(y)}w_0(y)-\Phi_k(x,y)\frac{\pa w_0}{\pa\nu}(y)\right\}ds(y)\cr
&=&-\int_{\pa D}\Phi_k(x,y)\frac{\pa w_0}{\pa\nu}(y)ds(y),\quad x\in \R^2\ba\ov{D}.
\enn
Letting $x$ approach the boundary, we have that
\be\label{w0-b}
S\frac{\pa w_0}{\pa\nu} = 0 \quad\mbox{and}\quad \Big(\frac{1}{2}I+K^{'}\Big)\frac{\pa w_0}{\pa\nu} = 0 \qquad \mbox{on}\, \pa D.
\en

Interchanging the order of integration, we observe that $S$ is self-adjoint and
$K$ and $K^{'}$ are adjoint with respect to the bilinear form
\be\label{bilinearform}
\langle\phi,\psi\rangle:=\int_{\pa D} \phi\psi ds.
\en
With the help of \eqref{wiv-b} and \eqref{w0-b}, we have that
\be\label{wiw=0}
\left\langle\wi{v},\frac{\pa w_0}{\pa\nu}\right\rangle
&=&\left\langle\Big(\frac{1}{2}I+K\Big)\wi{v} - S \frac{\pa\wi{v}}{\pa\nu},\frac{\pa w_0}{\pa\nu}\right\rangle\cr
&=&\left\langle\wi{v}, \Big(\frac{1}{2}I+K^{'}\Big)\frac{\pa w_0}{\pa\nu}\right\rangle - \left\langle \frac{\pa\wi{v}}{\pa\nu}, S\frac{\pa w_0}{\pa\nu}\right\rangle\cr
&=& 0.
\en

Define the single layer potential
\ben
\wi{w}(z):=\int_{\pa D}\Phi_k(y, z)\frac{\pa w_0}{\pa\nu}(y)ds(y), \quad z\in \R^2\ba\pa D.
\enn
Then $\wi{w}=0$ in $D_0$ by \eqref{wiw=0} and \eqref{wiv=Phi}. By analytic continuation, $\wi{w}=0$ in $\R^2\ba\ov{D}$. Taking the trace and using Lemma \ref{Uni-Interior},
$\wi{w}=0$ in $D$. Then $\frac{\pa w_0}{\pa\nu}=0$ on $\pa D$ by the jump relation of the single layer potential. Using the homogeneous Dirichlet boundary condition $w_0=0$ on $\pa D$, we have $w_0=0$ in $\R^2\ba\ov{D}$ by the Holmgren's theorem. This is a contradiction to the fact that $w_0$ is a non-trivial outgoing solution in $\R^2\ba\ov{D}$.
\end{proof}

Let $g_z^\eps$ be an approximate solution of
\begin{equation}\label{Near-Field-Equation}
(\mathcal{N}_k g)(x) = \Phi_k(x, z),\quad  x \in \pa \Om, z\in \R^2\ba\ov{D}.
\end{equation}
Due to Theorems \ref{pole-D} and \ref{LSM-D}, one can use $\|\mathcal{S}_k g^{\eps}_z\|_{H^1_{loc}(\R^2\ba{\ov{D}})}$ to characterize the Dirichlet poles. Based on the analysis in the next subsection for the Neumann case, $\|\mathcal{S}_k g^{\eps}_z\|_{H^1_{loc}(\R^2\ba{\ov{D}})}$ can also be used to characterize the Neumann poles.



\subsection{Neumann Case}\label{subsec-N}
In this subsection we consider Neumann boundary conditions.
For an outgoing solution $w\in H^1_{loc}(\R^2\ba{\ov{D}})$, the operator $\mathcal {G}_k: H^1_{loc}(\R^2\ba{\ov{D}}) \to L^2(\pa \Om)$ is now defined by $\mathcal {G}_k w=u_w|_{\pa \Om}$, where $u_w\in H^1(D)$ solves
\ben
\Delta u_w+k^2 u_w= 0 \quad\text{in } D\quad\mbox{and}\quad \frac{\pa u_w}{\pa\nu} = \frac{\pa w}{\pa\nu}  \quad\text{on } \pa D.
\enn
The near field operator $\mathcal{N}_k: L^2(\pa \Om) \to L^2(\pa \Om)$ has a factorization
     \be\label{Factorization-N}
     \mathcal{N}_k=\mathcal{G}_k\mathcal{S}_k;
     \en

\begin{theorem} \label{LSM-N}
Assume that $k\in C_{-}$ is not a Neumann pole. 
Let $z\in \R^2\ba\ov{D}$. Then for every $\eps>0$, there exists $g_\eps^z\in L^2(\pa\Om)$ such that
\ben
\lim_{\eps\rightarrow 0}\|\mathcal{N}_k g^{\eps}_z - \Phi_k(\cdot, z)\|_{L^2(\pa\Om)}=0\quad\mbox{and}\quad \|\mathcal{S}_k g^{\eps}_z\|_{H^1_{loc}(\R^2\ba{\ov{D}})} < M
\enn
for some constant $M>0$. 
\end{theorem}
\begin{proof}
     For $z\in\R^2\ba\ov{D}$,  $\Phi_k(\cdot, z)$ is a solution to the Helmholtz equation in $D$. Since $k\in {\mathbb C}_{-}$ is not a Neumann pole, the exterior Neumann boundary problem \eqref{Helmholtz}, \eqref{SoundHard}-\eqref{outgoing} is uniquely solvable. Denote by $w\in H^1_{loc}(\R^2\ba\ov{D})$ the outgoing solution with boundary data $\pa w/\pa\nu=\pa\Phi_k(\cdot, z)/\pa\nu$ on $\pa D$. Then $\mathcal {G}_k w=\Phi_k(\cdot, z)$, i.e., $\Phi_k(\cdot, z)\in \mathscr{R}(\mathcal {G}_k)$.
     The results then follows by a similar proof of Theorem \ref{LSM-D}.
\end{proof}

We have the following analogy of Theorem \ref{pole-D}.

\begin{theorem}\label{pole-N}
Assume that $k\in C_{-}$ is a Neumann pole and 
\be\label{Ng-Phi-N}
\lim_{\eps\rightarrow 0}\|\mathcal{N}_k g^{\eps}_z - \Phi_k(\cdot, z)\|_{L^2(\pa\Om)}=0.
\en
holds for all $z\in \R^2\ba\ov{D}$.
Then, for almost every $z\in \R^2\ba\ov{D}$, $\|\mathcal{S}_k g^{\eps}_z\|_{H^{1}_{loc}(\R^2\ba\ov{D})}$ can not be bounded as $\eps\rightarrow 0$.
\end{theorem}
\begin{proof}
Assume on the contrary that
\ben
\|\mathcal{S}_k g^{\eps_n}_z\|_{H^{1}_{loc}(\R^2\ba\ov{D})}\leq M
\enn
for some constant $M>0$, some sequence $\eps_n\rightarrow 0$ and all
$z\in D_0\subset \R^2\ba\ov{D}$, where $D_0$ has a positive measure. By selecting subsequences and relabeling, there exists a  weakly convergent sequence  $v_n:=\mathcal{S}_k g^{\eps_n}_z \rightharpoonup v$ for some $v\in H^{1}_{loc}(\R^2\ba\ov{D})$. Denote by $T: H^{1}_{loc}(\R^2\ba\ov{D}) \to H^{-1/2}(\pa D)$ the Neumann trace operator. Then $T v_n \rightharpoonup Tv$ in $H^{1/2}(\pa D)$.

Let $\wi{v}\in H^{1}_{loc}(\R^2\ba\ov{D})$ be an outgoing solution to the Helmholtz equation in $\R^2\ba\ov{D}$ with Neumann boundary value
$\frac{\pa\wi{v}}{\pa\nu}=Tv$ on $\pa D$.
This implies $T\wi{v}=Tv$ and therefore $T v_n \rightharpoonup T\wi{v}$ in $H^{-1/2}(\pa D)$.
Note that $\mathcal{G}_k = GT$, where $G: H^{-1/2}(\pa D) \to L^2(\pa \Om)$ is the compact operator defined by \eqref{G-Def}.
Furthermore, using the factorization $\mathcal{N}_k=\mathcal{G}_k\mathcal{S}_k$, we have that
\be\label{Gvn-Gv-N}
\mathcal{N}_k g^{\eps}_z = \mathcal{G}_k\mathcal{S}_k g^{\eps}_z = \mathcal{G}_k v_n = GT v_n \rightarrow GT \wi{v} = \mathcal{G}_k \wi{v} \quad\mbox{in}\, L^2(\pa\Om).
\en
Denote by $u_z$ the solution of
\ben
\Delta u_z +k^2 u_z = 0 \quad\mbox{in}\, D, \qquad \frac{\pa u_z}{\pa\nu}=T\wi{v}\quad\mbox{on}\, \pa D.
\enn
Then $\mathcal{G}_k \wi{v} = u_z$ in $L^2(\pa\Om)$. Combination of \eqref{Gvn-Gv} and \eqref{Ng-Phi} yields that $u_z=\Phi(\cdot,z)$ on $\pa\Om$. Using Lemma \ref{Uni-Interior} and analytic continuation, we have that $u_z=\Phi(\cdot, z)$ in $D$ and thus
\be\label{wiv=Phi-N}
T\wi{v}=\frac{\pa\Phi}{\pa\nu}(\cdot, z)\quad\mbox{on}\, \pa D.
\en

Recalling the double-layer operator $K$ defined by \eqref{K}, we introduce the normal derivative operator $\mathbb{T}$ of $K$ by
\ben
\mathbb{T}\phi(x):=\frac{\pa}{\pa\nu(x)}\int_{\pa D}\frac{\pa\Phi_k(x,y)}{\pa\nu(y)}\phi(y)ds(y), \quad x\in\pa D.
\enn
Then $\mathbb{T}$ is self-adjoint with respect to the bilinear form \eqref{bilinearform}.
Since $\wi{v}$ is an outgoing solution, we obtain that
\be\label{wiv-N}
\wi{v}(x)=\int_{\pa D}\left\{\frac{\pa \Phi_k(x,y)}{\pa\nu(y)}\wi{v}(y)-\Phi_k(x,y)\frac{\pa \wi{v}}{\pa\nu}(y)\right\}ds(y),\quad x\in \R^2\ba\ov{D}.
\en
Using the jump relation, \eqref{wiv-N} implies
\be\label{wiv-b-N}
T\wi{v} = \mathbb{T}\wi{v}|_{\pa D}+\Big(\frac{1}{2}I-K^{'}\Big)\frac{\pa\wi{v}}{\pa\nu}\Big|_{\pa D}.
\en

Since $k\in {\mathbb C}_{-}$ is a Neumann pole, we take $w_0$ as a non-trivial outgoing solution in $\R^2\ba\ov{D}$ with homogeneous Neumann boundary data $\frac{\pa w_0}{\pa\nu}= 0$ on $\pa D$.
By the outgoing condition again, 
\ben
w_0(x)
&=&\int_{\pa D}\left\{\frac{\pa \Phi_k(x,y)}{\pa\nu(y)}w_0(y)-\Phi_k(x,y)\frac{\pa w_0}{\pa\nu}(y)\right\}ds(y)\cr
&=&\int_{\pa D}\frac{\pa \Phi_k(x,y)}{\pa\nu(y)}w_0(y)ds(y),\quad x\in \R^2\ba\ov{D}.
\enn
Letting $x$ approach the boundary, we have that
\be\label{w0-b-N}
\mathbb{T}w_0 = 0 \quad\mbox{and}\quad \Big(\frac{1}{2}I-K\Big) w_0 = 0 \qquad \mbox{on}\, \pa D.
\en

With the help of \eqref{wiv-b-N} and \eqref{w0-b-N}, noting that $\mathbb{T}$ is self-adjoint and $K$ and $K^{'}$ are adjoint with respect to the bilinear form \eqref{bilinearform}, it holds that that
\be\label{wiw=0-N}
\left\langle T\wi{v}, w_0\right\rangle
&=&\left\langle\mathbb{T}\wi{v}|_{\pa D}+\Big(\frac{1}{2}I-K^{'}\Big)\frac{\pa\wi{v}}{\pa\nu}\Big|_{\pa D},\, w_0\right\rangle\cr
&=&\left\langle\wi{v}+\frac{\pa\wi{v}}{\pa\nu},\, \mathbb{T}w_0+\Big(\frac{1}{2}I-K\Big)w_0\right\rangle\cr
&=&0.
\en

Define the double-layer potential
\ben
\wi{w}(z):=\int_{\pa D}\frac{\pa\Phi_k(y, z)}{\pa\nu(y)}  w_0(y) ds(y), \quad z\in \R^2\ba\pa D.
\enn
Then $\wi{w}=0$ in $D_0$ by \eqref{wiw=0-N} and \eqref{wiv=Phi-N}. By analytic continuation, $\wi{w}=0$ in $\R^2\ba\ov{D}$. Taking the Neumann trace and using Lemma \ref{Uni-Interior},
$\wi{w}=0$ in $D$. Then $w_0=0$ on $\pa D$ by the jump relation of the double-layer potential. Due to the homogeneous Neumann boundary condition $\frac{\pa w_0}{\pa\nu}=0$ on $\pa D$, $w_0=0$ in $\R^2\ba\ov{D}$ by the Holmgren's theorem. This is a contradiction to the fact that $w_0$ is non-trivial outgoing solution in $\R^2\ba\ov{D}$.
\end{proof}

\section{Numerical Examples}
We now present some examples to validate the inside-out duality developed in the previous section, i.e., characterize the scattering poles using the interior scattering problems. More precisely, we solve the near field equation \eqref{Near-Field-Equation} for some point $z$ outside the obstacle and check the norm of the approximate solutions for various $k$'s. We expect that the norm is large when $k$ is a pole. The "exact" poles are computed by a finite element DtN method proposed in \cite{XiGongSun2024}.

For a sound-soft obstacle and a fixed $k$, we use a linear Lagrange finite element method to compute the scattering data in $D$ with the Dirichlet data $f(x) = -\Phi_k(x, y), y \in \partial \Omega, x \in D$. To discretize \eqref{Near-Field-Equation}, we collect data $u^s_k(x_j, y_i)$, $x_i, y_j \in \partial \Omega,  i,j=1, \ldots, N$, which are computed as follows. For each $y_j$, one computes $u^s_k(x_j, y_i), i=1, \ldots, N$.
In other words, one solves $N$ Helmholtz equations for each $k$. We use the same set of discretization points on $\partial \Omega$ for $g(y)$ and denote the vector by ${\boldsymbol g}^k$ with ${\boldsymbol g}^k_i=g(y_i), i=1, \ldots, N,$ for some $z \in \mathbb R^2 \setminus \overline{D}$. The discrete version of \eqref{Near-Field-Equation}
becomes a system of equations
\begin{equation}\label{DNF}
\sum_{i=1}^N \omega_i u_k^s(x_j, y_i) {\boldsymbol g}^k_i = \Phi_k(x_j, z), \quad j=1, \ldots, N,
\end{equation}
where $\omega_i$'s are the weights of the quadrature rule for the integral over $\partial \Omega$. If $\partial \Omega$ is a circle, then one can simply use the trapezoidal rule such that $\omega_1=\omega_2 = \ldots =\omega_N$.

Let $S \subset \mathbb C_-$ be a region of interest. Let $|{\boldsymbol g}^k|$ be the norm of the solution of \eqref{DNF}. We expect that $|{\boldsymbol g}^k|$ becomes larger as $k$ approaches a scattering pole. We compute $|{\boldsymbol g}^k|$'s for a set $\{k_i\} \subset S$ of sampling points that covers $S$ and look for locations where $|{\boldsymbol g}^k|$'s are large. The case for a sound-hard obstacle is treated similarly.

\begin{remark}
We use $l_2$ norm of ${\boldsymbol g}^k$, not the respective norms in $H^{1/2}(\partial D)$ and $H_{loc}^1(\mathbb R^2 \setminus \overline{D})$. Numerical examples in this paper and results for the linear sampling method in literature suggest that such a simplification works well.
\end{remark}

\begin{remark}
For a fixed $z \in \mathbb R^2 \setminus \overline{D}$, $\Phi(\cdot, z)|_{\partial \Omega}$ is a smooth function of $k$. In stead of solving the integral equation, one can also use the condition number of the matrix $U(k)$ where $U_{i,j}(k) = u^s_k(x_i, y_j)$ to characterize the scattering poles. 
\end{remark}

\begin{remark}
Numerical experiments indicate that the choice $z \in \mathbb R^2 \setminus \overline{D}$ does not affect the result significantly as long as $z$ is not too far away from $\partial D$.
\end{remark}

\subsection{Sound-soft Obstacles}

{\bf Example 1.}
Let $D$ be the unit disk. Using separation of variables, it can be derived that the scattering poles are the zeros of Hankel functions $H_n^{(1)}(k), n \ge 2$ (see Table 1 of \cite{MaSun2023}). In particular, two poles with smallest norms are given by $k_1=0.4295-1.2814i$ and $k_2=1.3080-1.6818i$.

Let $\partial \Omega$ be a circle with radius $0.7$ inside $D$. We choose $40$ points $y_j, j=1, \ldots, 40,$ uniformly on $\partial \Omega$. The same set of points is also used for $x_i$'s. Let $z=(1,1)$, outside $D$. We generate a triangular mesh with mesh size $h\approx 0.025$ for the finite element method to compute $u^s_k(x_i, y_j)$. Then we solve \eqref{DNF} for ${\boldsymbol g}^k$ using the Tikhonov regularization. 

Let $S=[0.42, 0.44]\times[-1.29,-1.27]$ be a small region containing $k_1$. We choose $41 \times 41$ $k$'s uniformly in $S$ and compute $|{\boldsymbol g}^k|$'s. We expect that $|{\boldsymbol g}^k|$'s are large when the sampling points $k$'s are close to $k_1$. In deed, the plot of $|{\boldsymbol g}^k|$'s has a spike at $0.4285-1.2805i$, which is shown in the left picture of Fig.~\ref{SScircle}. 

Similarly, let $S=[1.30, 1.32]\times[-1.29,-1.27]$ be a small region containing $k_2$. Again, we choose $41 \times 41$ $k$'s uniformly in $S$ and compute $|{\boldsymbol g}^k|$'s. We plot $|{\boldsymbol g}^k|$'s (right picture of Fig.~\ref{SScircle}) and a spike is found at $1.3080-1.6820 i$. Both values are good approximations of the exact poles.

\begin{figure}[ht]
\begin{center}
\begin{tabular}{cc}
\resizebox{0.5\textwidth}{!}{\includegraphics{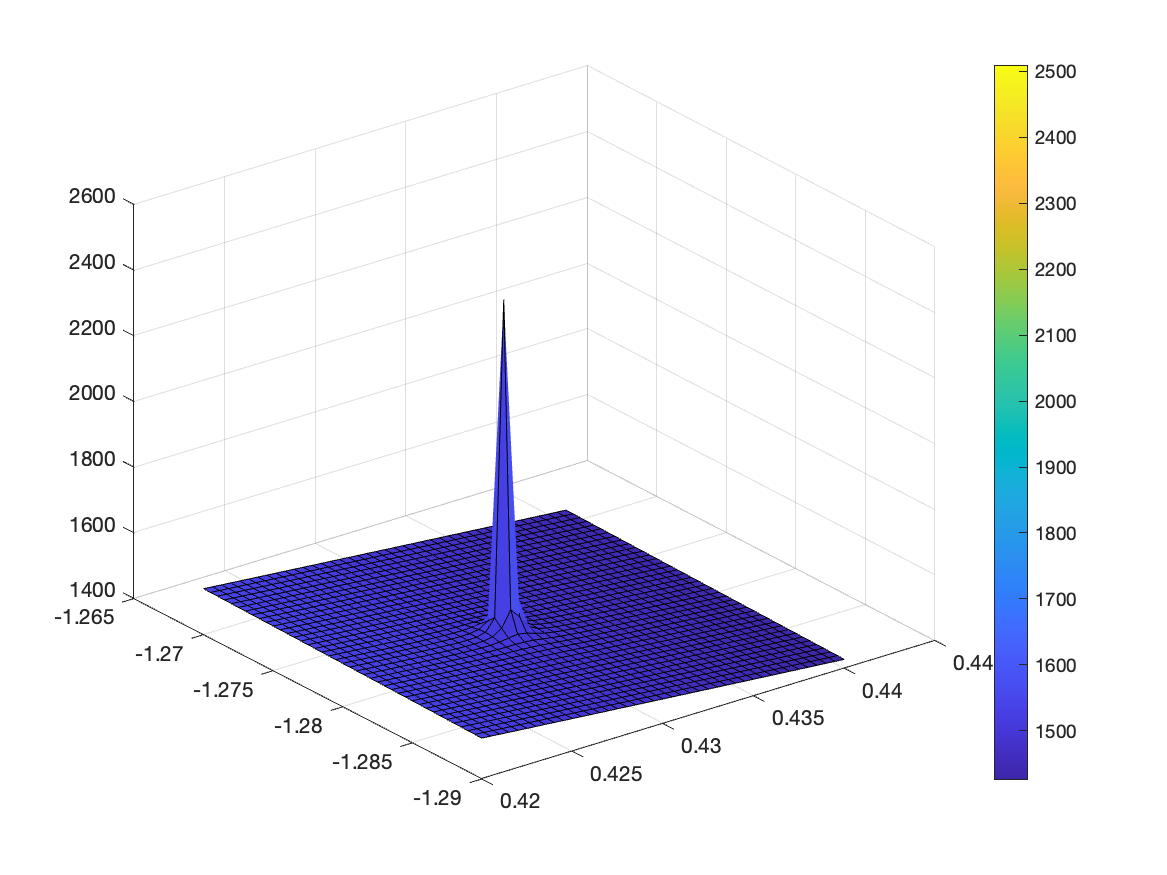}}&
\resizebox{0.5\textwidth}{!}{\includegraphics{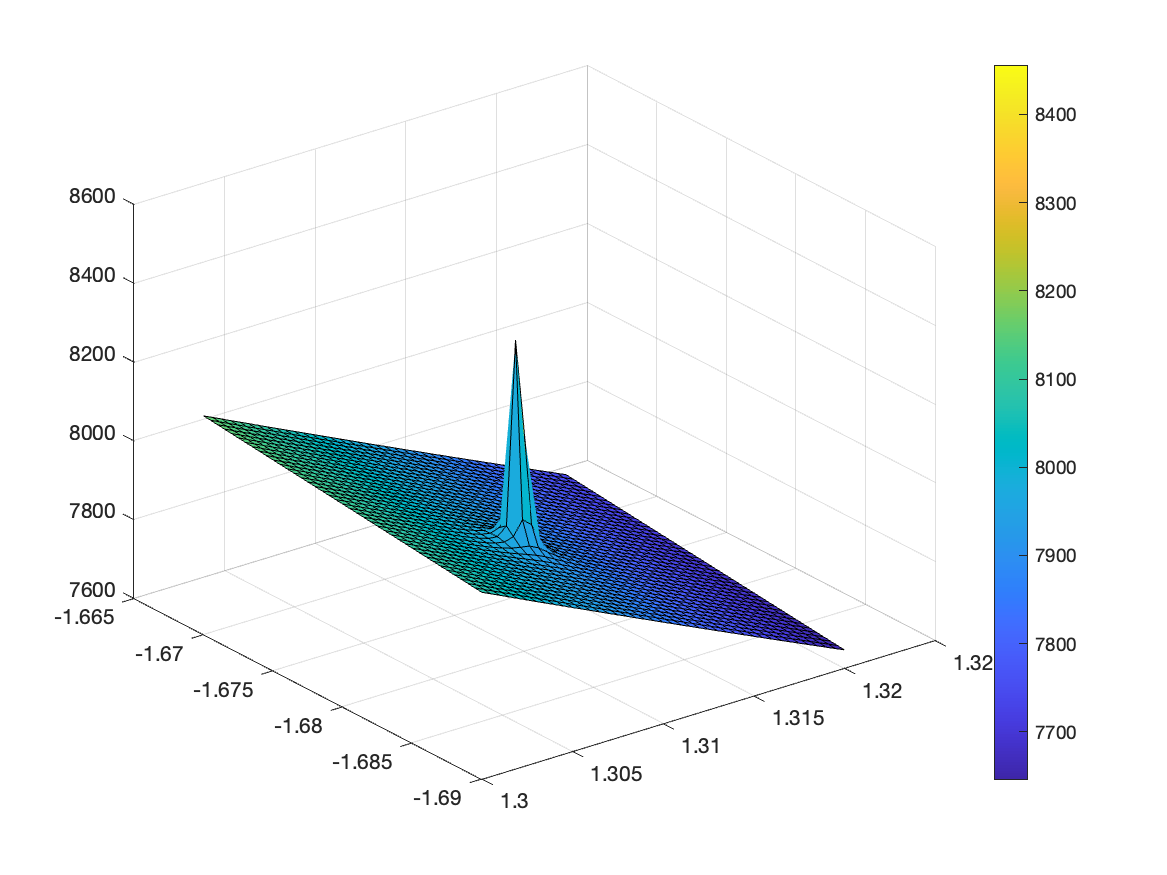}}
\end{tabular}
\end{center}
\caption{Plot of $|{\boldsymbol g}^k|$ for the sound-soft disk.  Left: $S=[0.42, 0.44]\times[-1.29,-1.27]$. Right: $S=[1.30, 1.32]\times[-1.29,-1.27]$.}
\label{SScircle}
\end{figure}

\noindent {\bf Example 2.}
Let $D$ be the ellipse whose boundary is given by $\frac{x^2}{1.3^2}+\frac{y^2}{0.7^2}=1$. There are two poles 
$0.4586 - 1.2774i$ and $0.4315 - 1.3062i$ close to $0.44-1.29i$.
We take $S=[0.40, 0.48]\times[-1.33,-1.26]$ and choose $\partial \Omega$ as a circle with radius $0.5$ inside $D$. 
 Let $z=(1.4, 1.0)$ outside $D$. We choose $161 \times 161$ points uniformly for $S$ and compute $|{\boldsymbol g}^k|$.
 The plot of $|{\boldsymbol g}^k|$'s against the sampling points are shown in the left picture of Fig.~\ref{SSellipse}. There are two spikes which locate at
 $0.4215-1.2975i$ and $0.4510-1.2710i$ and are close to the exact poles.
 
 There are also two scattering poles at $1.3378 - 1.6671i$ and $1.3268 - 1.6739i$. Let $S=[1.28, 1.36]\times[-1.71,-1.63]$. We choose $161 \times 161$ $k$'s uniformly for $S$ and check $|{\boldsymbol g}^k|$.
 The plot of $|{\boldsymbol g}^k|$'s against the sampling points are shown in the right picture of Fig.~\ref{SSellipse}. Two spikes locate at
 $1.3345-1.6795i$ and $1.343-1.6710i$, again close to the exact poles.

\begin{figure}[ht]
\begin{center}
\begin{tabular}{cc}
\resizebox{0.5\textwidth}{!}{\includegraphics{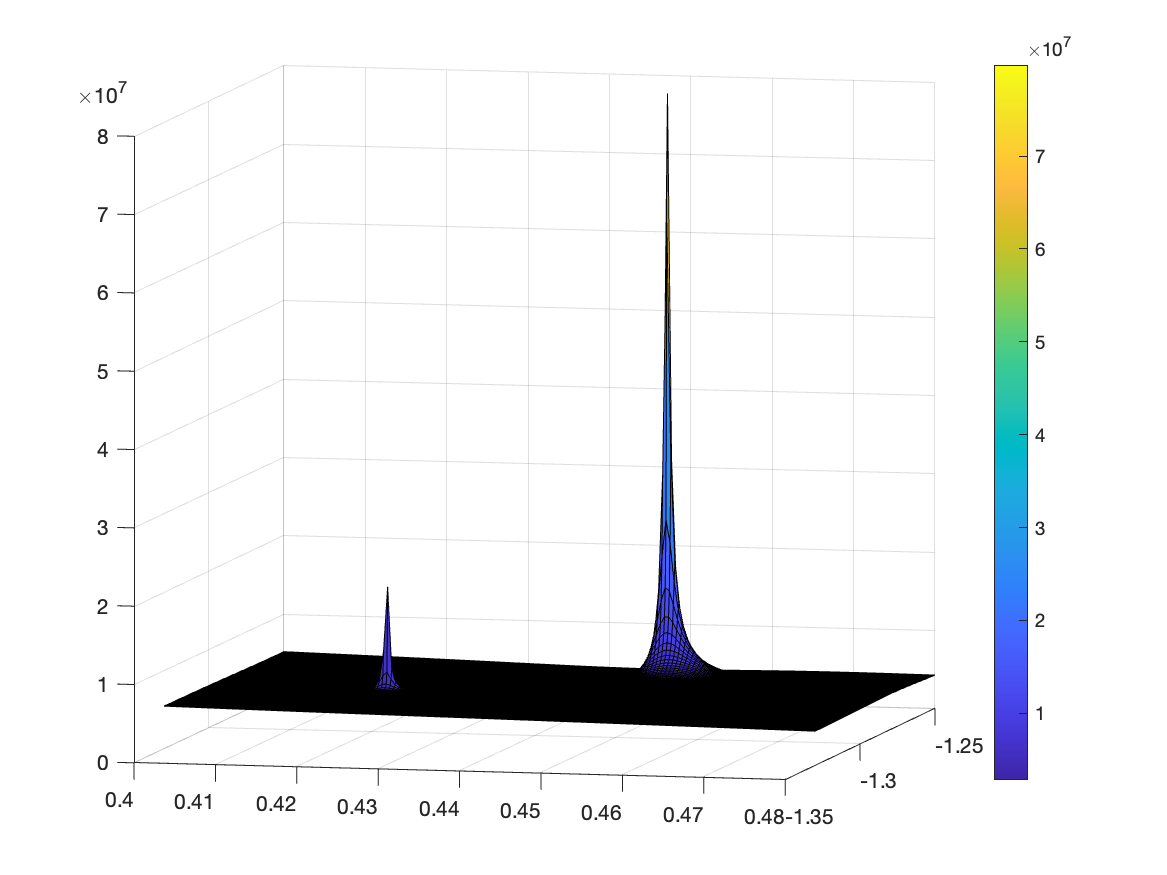}}&
\resizebox{0.5\textwidth}{!}{\includegraphics{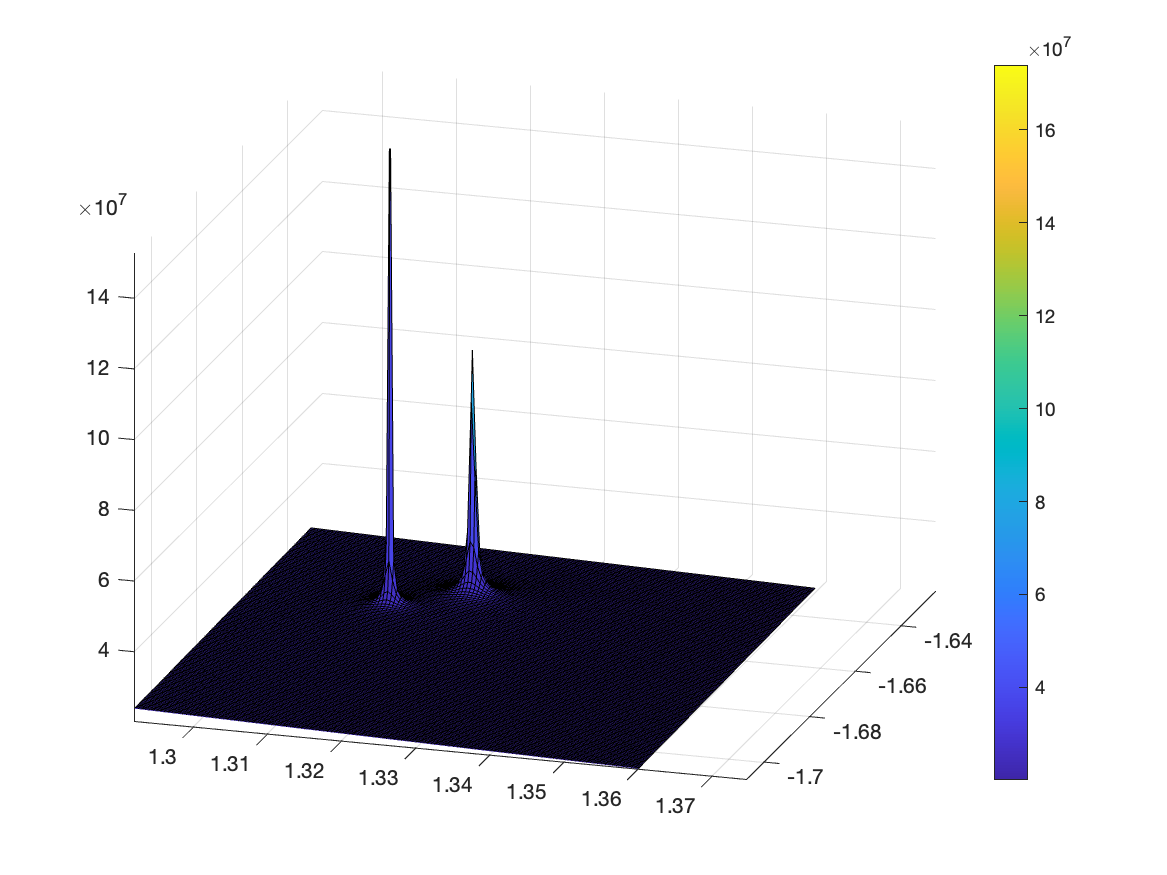}}
\end{tabular}
\end{center}
\caption{Plot of $|{\boldsymbol g}^k|$ for the sound-soft ellipse. Left: $S=[0.40, 0.48]\times[-1.33,-1.26]$. Right: $S=[1.28, 1.36]\times[-1.71,-1.63]$.}
\label{SSellipse}
\end{figure}

\subsection{Sound-hard Obstacles}

\noindent {\bf Example 3.} Let $D$ be the unit disk. Two exact scattering poles are given by
$0.5012 - 6.4355i$ and $1.4344 - 0.8345i$. Let $\partial \Omega$ be a circle with radius $0.7$ inside $D$. Again, we choose $40$ points $y_j, j=1, \ldots, 40,$ uniformly on $\partial \Omega$. The same set of points is also used for $x_i$'s. Let $z=(0,1.5)$, outside $D$. Set $S=[0.49, 0.51]\times[-0.65,-0.63]$. We choose $41 \times 41$ $k$'s uniformly in $S$ and compute $|{\boldsymbol g}^k|$'s. The plot of $|{\boldsymbol g}^k|$'s has a spike at $0.5018-0.6442i$, which is shown in the left picture of Fig.~\ref{NCcircle}. Then we set $S=[1.42, 1.44] \times [-0.85, -0.83]$. The plot of $|{\boldsymbol g}^k|$'s is the right picture of Fig.~\ref{NCcircle}, which has a spike at $1.4360-0.8355i$.

\begin{figure}[ht]
\begin{center}
\begin{tabular}{cc}
\resizebox{0.5\textwidth}{!}{\includegraphics{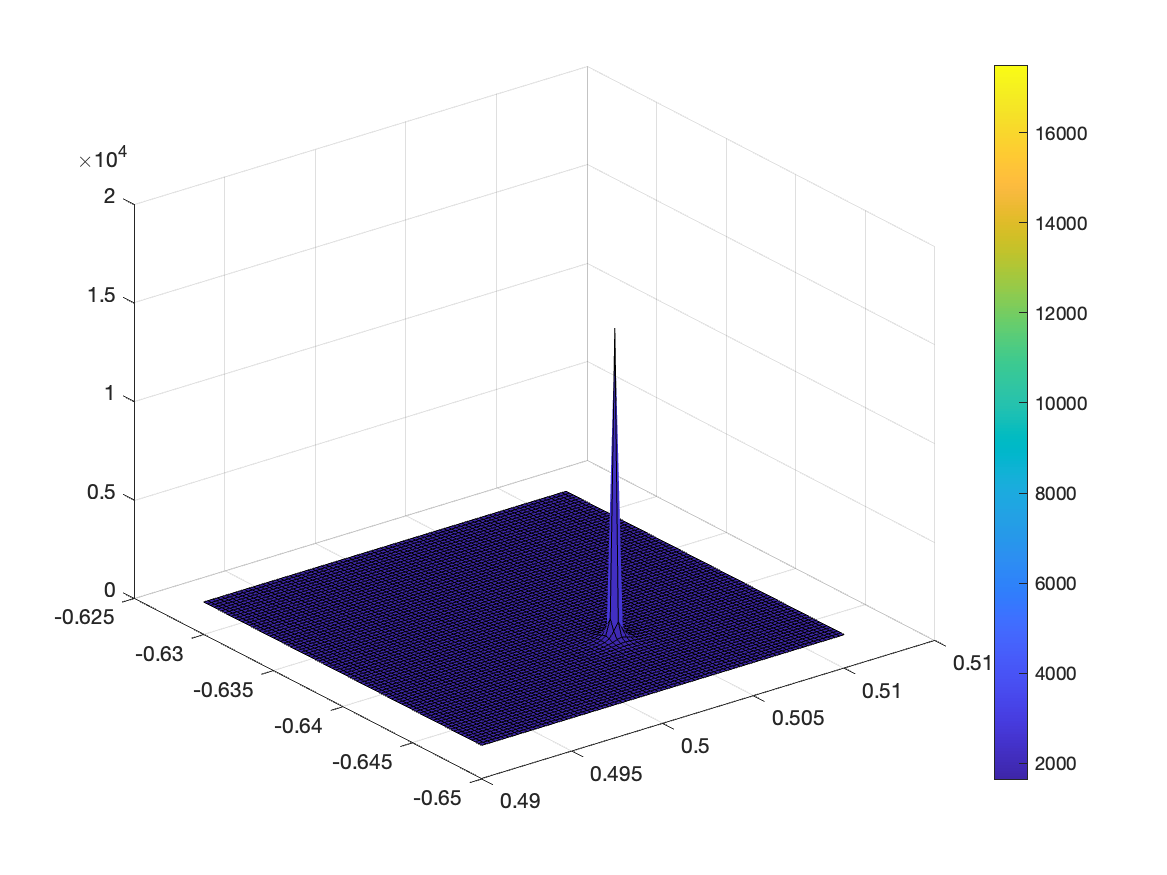}}&
\resizebox{0.5\textwidth}{!}{\includegraphics{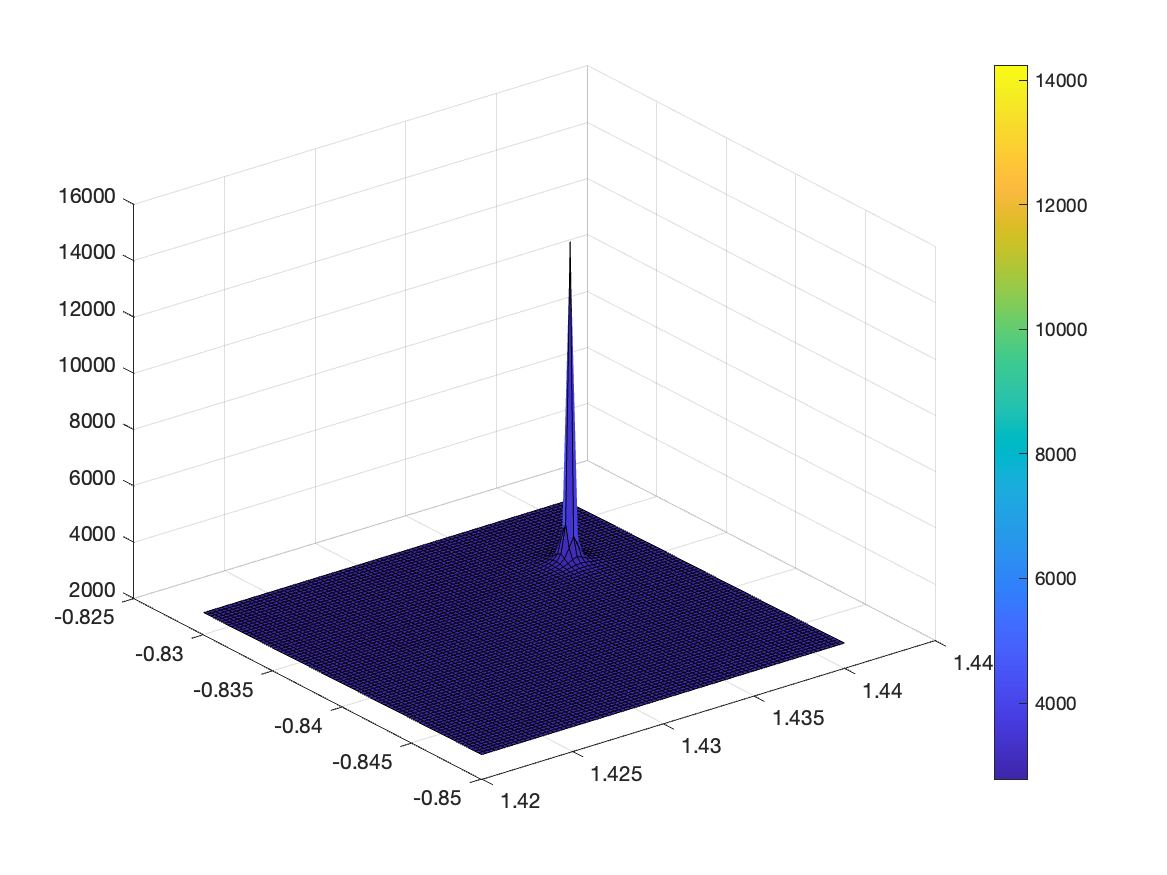}}
\end{tabular}
\end{center}
\caption{Plot of $|{\boldsymbol g}^k|$ for the sound-hard disk. Left: $S=[0.49, 0.51]\times[-0.65,-0.63]$. Right:  $S=[1.42, 1.44] \times [-0.85, -0.83]$.}
\label{NCcircle}
\end{figure}

\noindent{\bf Example 4.} Let $D$ be the ellipse whose boundary is given by $\frac{x^2}{1.3^2}+\frac{y^2}{0.7^2}=1$. Two exact scattering poles are $0.5388 - 0.5623i$ and $1.4436 - 0.7840i$. Let $\partial \Omega$ be a circle with radius $0.5$ inside $D$. We choose $40$ points $y_j, j=1, \ldots, 40,$ uniformly on $\partial \Omega$. The same set of points is also used for $x_i$'s. Let $z=(0,1.5)$, outside $D$, and $S=[0.53, 0.55]\times[-0.57,-0.55]$. We choose $41 \times 41$ $k$'s uniformly in $S$ and compute $|{\boldsymbol g}^k|$'s. The plot of $|{\boldsymbol g}^k|$'s has a spike at $0.5435-0.5548i$, which is shown in the left picture of Fig.~\ref{NCellipse}. Then we set $S=[1.44, 1.46] \times [-0.80, -0.78]$. The plot of $|{\boldsymbol g}^k|$'s is the right picture of Fig.~\ref{NCellipse}, which has a spike at $1.4470-0.7860i$.
 
 \begin{figure}[ht]
\begin{center}
\begin{tabular}{cc}
\resizebox{0.5\textwidth}{!}{\includegraphics{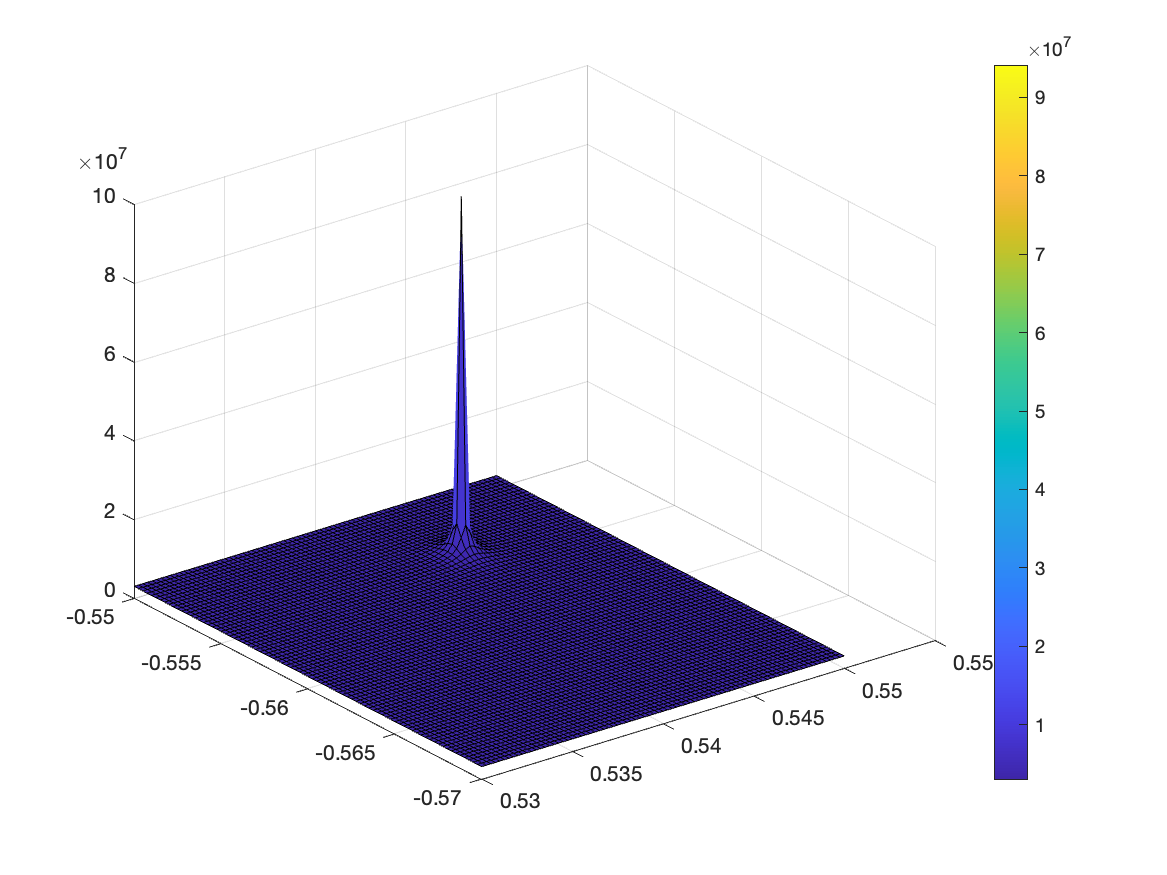}}&
\resizebox{0.5\textwidth}{!}{\includegraphics{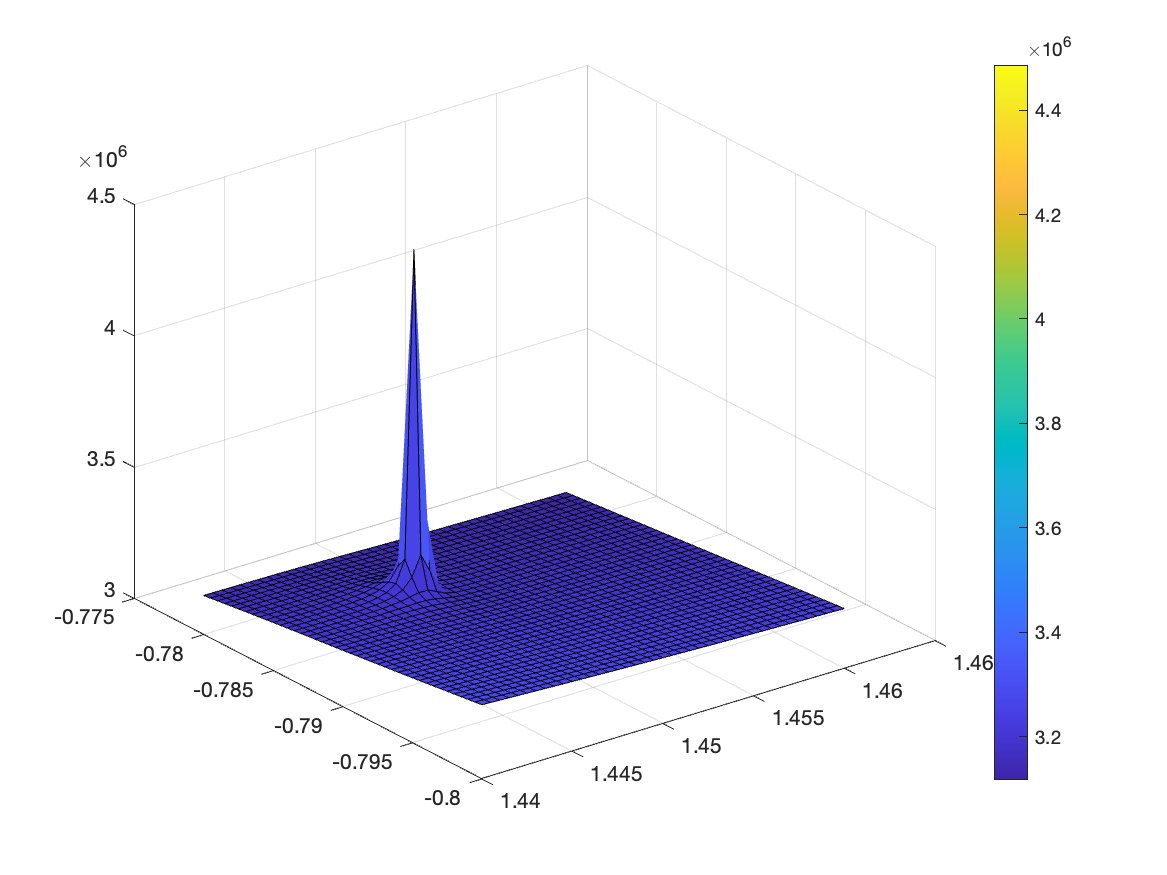}}
\end{tabular}
\end{center}
\caption{Sound-hard ellipse. Plot of $|{\boldsymbol g}^k|$ in different regions. Left: Region containing $k_1$. Right: Region containing $k_2$.}
\label{NCellipse}
\end{figure}

\subsection{Lipschitz Domains}
The above examples show that the characterization for scattering poles of smooth obstacles is rather accurate. However, similar accuracy is not achieved for domains with corners. Consider the sound-soft unit square $D=(-1/2, 1/2) \times (-1/2, 1/2)$. The exact scattering pole with smallest norm is $k_1 \approx 0.6747 - 2.0177i $.

Let $\partial \Omega$ be a circle with radius $0.3$ in $D$ and we carry out similar computation as the above examples. The spike of $|{\boldsymbol g}^k|$ appears at $0.7180-2.1440i$. To gain some understanding of this, we consider a series of domains with continuous boundaries approximating the unit square. In particular, we use quarters of the circles with radius $r = 0.2, 0.1, 0.04, 0.02, 0.002$ to replace the corners of the unit square (see Fig.~\ref{Squares}). 

The smallest scattering poles for the above domains are $0.7210 - 2.1537i$, $0.6920 - 2.0690i$, $0.6795 - 2.0324i$, $0.6764 - 2.0235i$, $0.6746 - 2.0181i$. These values are getting closer to a pole $0.6747 - 2.0177i$ of the unit square as $r$ becomes smaller. It shows that as the domains get closer to the unit square, the smallest scattering poles are also closer to the smallest pole of the unit square. The spike at $0.7180-2.1440i$ is close to a pole of the first domain in Fig.~\ref{Squares}.


\begin{figure}[ht]
\begin{center}
\begin{tabular}{ccc}
\resizebox{0.33\textwidth}{!}{\includegraphics{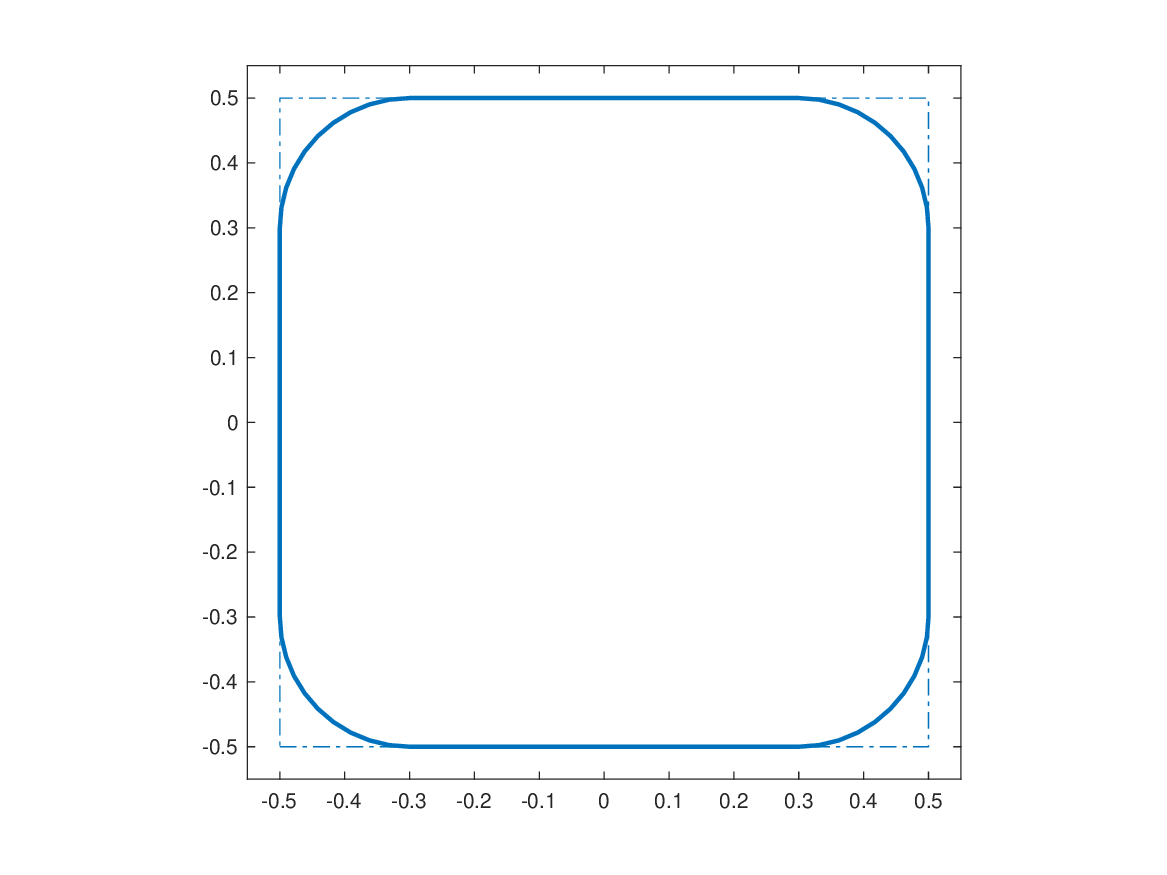}}&
\resizebox{0.33\textwidth}{!}{\includegraphics{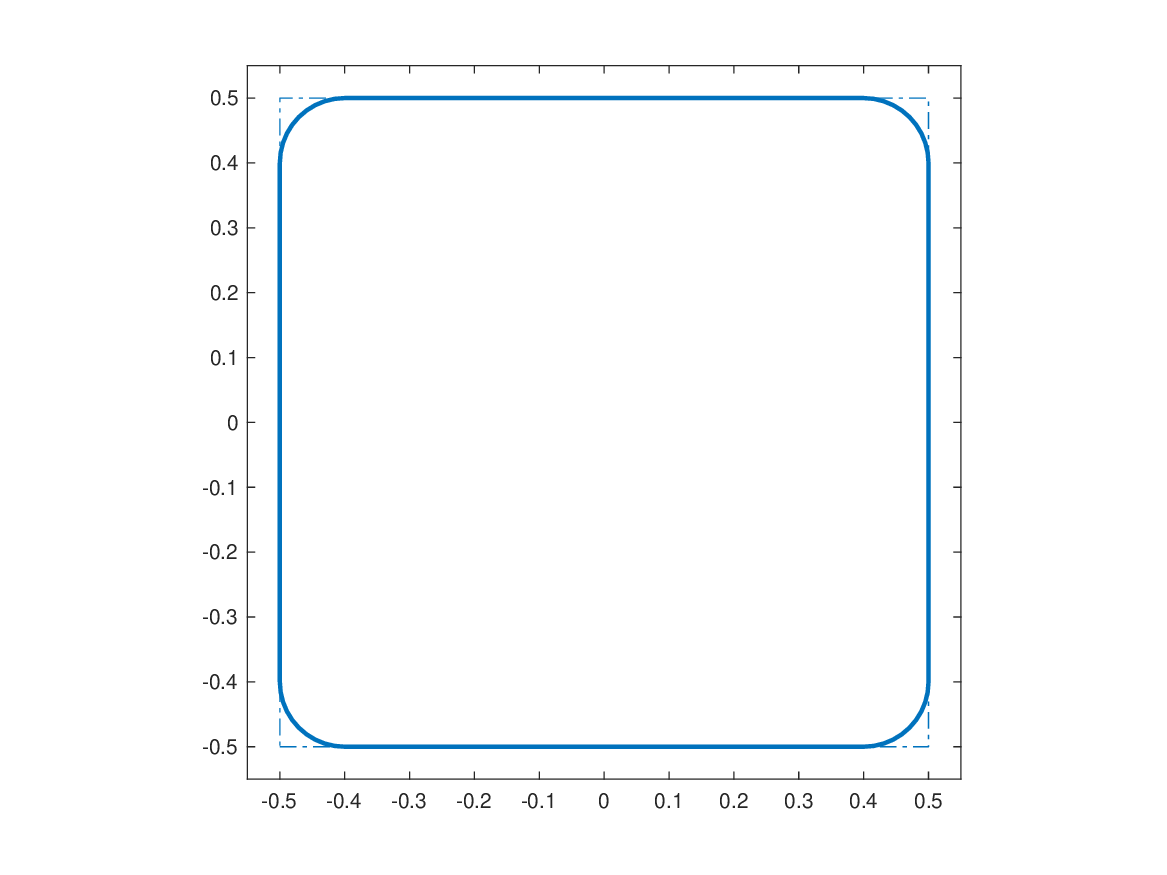}}&
\resizebox{0.33\textwidth}{!}{\includegraphics{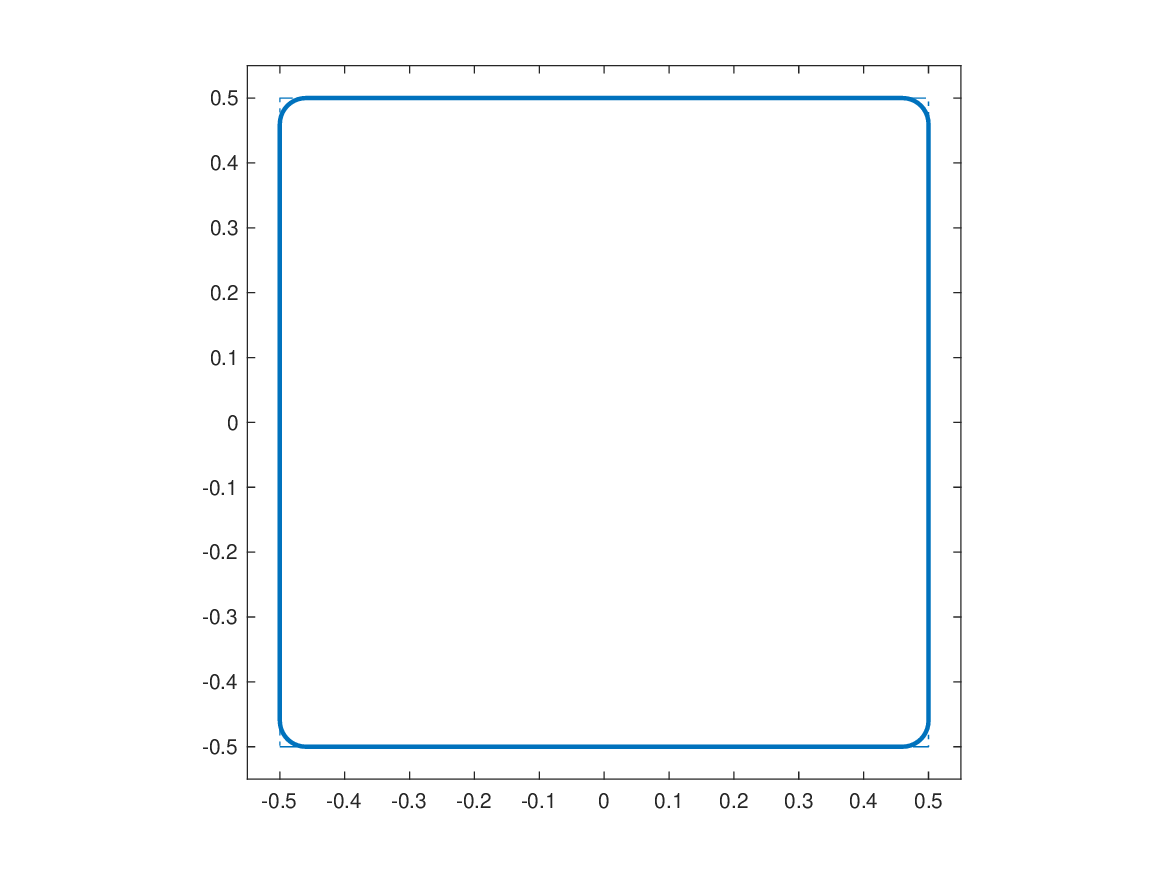}}\\
\resizebox{0.33\textwidth}{!}{\includegraphics{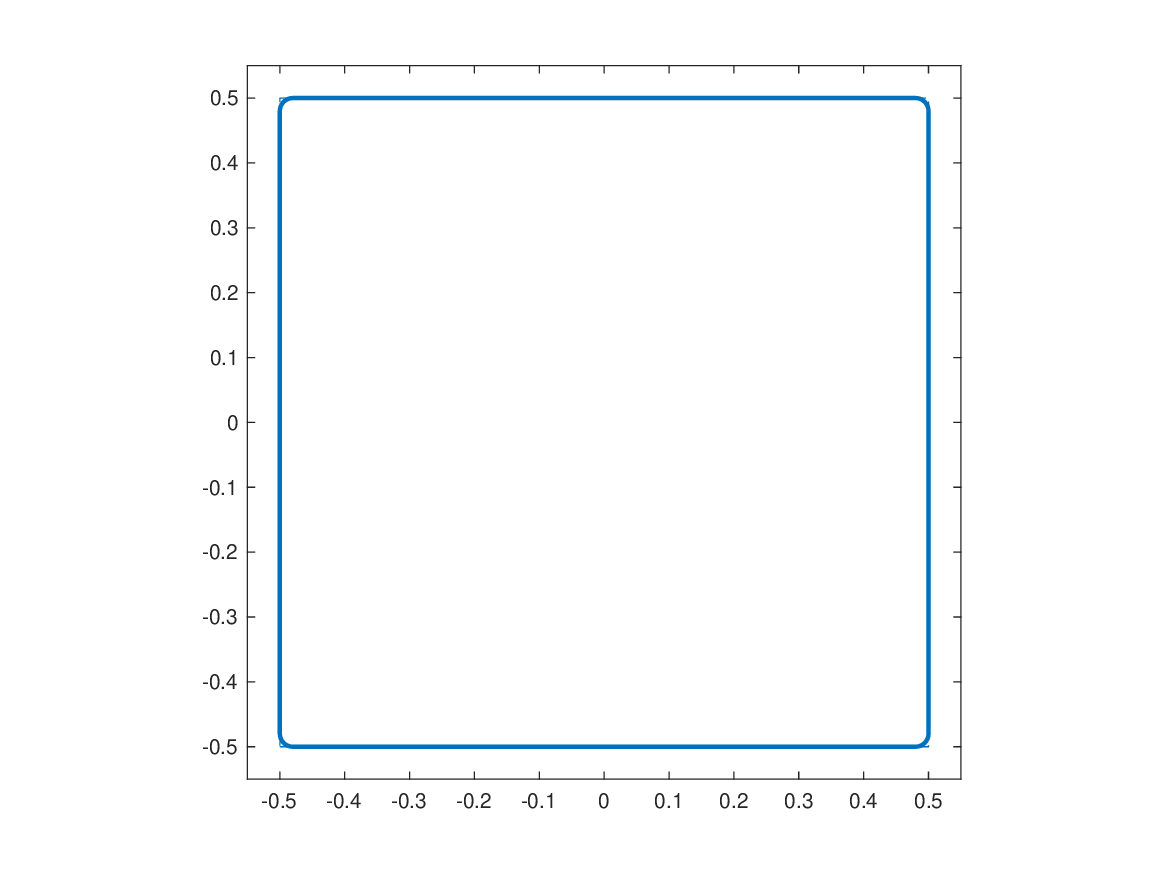}}&
\resizebox{0.33\textwidth}{!}{\includegraphics{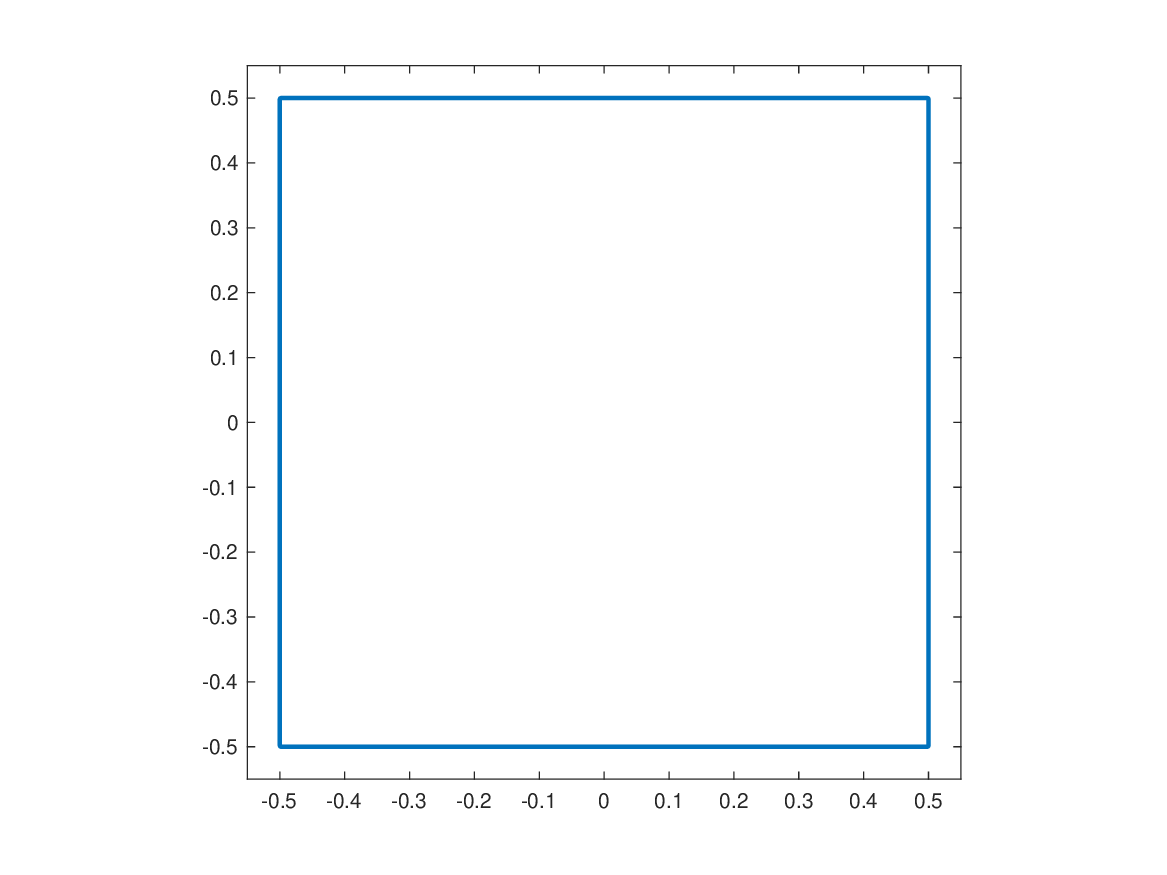}}
\end{tabular}
\end{center}
\caption{A series of domains with continuous boundaries approximating the unit square. The corners of the unit square are replaced with the corresponding quarters of a circle with radius $r$. Top left: $r=0.2$. Top middle: $r=0.1$. Top right: $r=0.04$. Bottom left: $r=0.02$. Bottom middle: $r=0.002$.}
\label{Squares}
\end{figure}

One question is that whether the information of a corner can be recovered using the interior scattering problem when the (complex) wavenumber is rather small in norm. Nonetheless, it remains interesting to investigate whether scattering poles of domains with corners can be characterized accurately using interior scattering problems.

   


\section{Conclusions and Future Work}
This paper investigates the scattering problems of sound-soft and sound-hard obstacles for complex wavenumbers. We establish two generalized Rellich's lemmas and use them to prove some uniqueness theorems for the inverse scattering problems. Additionally, we explore the inside-out duality, demonstrating that the scattering poles can be characterized by the injectivity of a near-field operator, which is defined using the solutions of interior scattering problems.

Numerical experiments suggest that scattering poles can be characterized with reasonable accuracy for smooth obstacles. However, the accuracy decreases for non-smooth obstacles. We believe that the complexity of the interior scattering problem is the primary source of this inaccuracy, as the scattered field conveys limited information about the corners. This limitation might affect the effectiveness of using the inside-out duality to estimate scattering poles. Nevertheless, it is worth exploring the possibility of improving accuracy by modifying the current algorithm or developing new algorithms.

The properties of scattering problems for complex wavenumbers are important research topics. We plan to extend our study to include inhomogeneous media and electromagnetic or elastic waves. Scattering resonances in periodic structures warrant investigation as well.

\section*{Acknowledgement}
The research of X. Liu is supported by the National Key R\&D Program of China under grant 2023YFA1009300, the NNSF of China under grant 12371430 and Beijing Natural Science Foundation Z200003. The research of J. Sun is funded in part by Simons Foundation Grant 711922 and NSF Grant 2109949. The research of L. Zhang is supported by the NNSF (12271482), Zhejiang Provincial NSFC (LZ23A010006; LY23A010004).


\begin{thebibliography}{} 
\label{bbiibb}
\bibitem{AS}M. Abramowitz and I.A. Stegun, Handbook of Mathematical Functions With Formulas, Graphs and Mathematical Tables, Dover, New York, 1965.


\bibitem{CCC2008} F. Cakoni, M. Çay\"{o}ren and D. Colton, {\em Transmission eigenvalues and the nondestructive testing of dielectrics.} Inverse Problems 24 (2008), no. 6, 065016, 15 pp

\bibitem{CCH2010} F. Cakoni, D. Colton and H. Haddar, {\em On the determination of Dirichlet and transmission eigenvalues from far field data.} C.R. Acad. Sci. Paris, Ser. 1 348, (2010), 379-383.

\bibitem{CCH2020} F. Cakoni, D. Colton and H. Haddar, {\em A duality between scattering poles and transmission eigenvalues in scattering theory.} Proc. A. 476 (2020), no. 2244, 20200612, 19 pp.

\bibitem{CHZ2024} F. Cakoni, H. Haddar and Z. Dana,  {\em An algorithm for computing scattering poles based on dual characterization to interior eigenvalues.} Proc. A. 480 (2024), 20240015. 

\bibitem{CK-integral} D. Colton and R. Kress, Integral Equation Methods in Scattering Theory. John Wiley \& Sons, Inc., New York, 1983.

\bibitem{Dolph} C. L. Dolph, {\em The integral equation method in scattering theory.} In: Problems in Analysis (Gunning, ed). Princeton University Press, Princeton, 201-227 (1970).

\bibitem{ColtonKress2019}
D. Colton and R. Kress, Inverse Acoustic and Electromagnetic Scattering Theory, 4th ed., Springer, Cham, 2019.



\bibitem{DyatlovZworski2019}
S. Dyatlov and M. Zworski, Mathematical Theory of Scattering Resonances. American Mathematical Society, Providence, RI, 2019.


\bibitem{KellerRubinowGoldstein1963}
J.B. Keller, S.I. Rubinow and M. Goldstein, {\em Zeros of Hankel functions and poles of scattering amplitudes.} J. Math. Phys.
4 (1963) 829-832.

\bibitem{KL2013} A. Kirsch and A. Lechleiter, {\em The inside-outside duality for scattering problems by inhomogeneous media.} Inverse Problems 29 (2013), no. 10, 104011, 21 pp.

\bibitem{Labreuche1998}
C. Labreuche, {\em Uniqueness and stability of the recovery of a sound soft obstacle from a knowldege of its scattering resonances.}  Comm. Partial Differential Equations 23 (1998), no. 9-10, 1719-1748. 

\bibitem{LaxPhillips1989} P.D. Lax and R.S. Phillips, Scattering Theory, 2nd ed., Academic Press, Inc.,
Boston, MA, 1989.

\bibitem{LS2014} X. Liu and J. Sun, {\em Reconstruction of Neumann eigenvalues and support of sound-hard obstacles.} Inverse Problems 30 (2014), no. 6, 065011, 17 pp.

\bibitem{MaSun2023} Y. Ma and J. Sun, {\em Computation of scattering poles using boundary integrals.} Appl. Math. Lett. 146 (2023), Paper No. 108792, 7 pp.

\bibitem{Melrose} R.B. Melrose, Geometric Scattering Theory. Cambridge University Press, Cambridge 1995.

\bibitem{SjostrZworski1991} J. Sj\"{o}str and M. Zworski, {\em Complex scaling and the distribution of scattering poles.} J. Amer. Math. Soc. 4(4) (1991), 729-769.

\bibitem{SteinbachUnger2012}
O. Steinbach \& G. Unger, Convergence analysis of a Galerkin boundary element method for the Dirichlet Laplacian
eigenvalue problem, SIAM J. Numer. Anal. 50 (2) (2012) 710-728.

\bibitem{S2011} J. Sun, {\em Estimation of transmission eigenvalues and the index of refraction from Cauchy data.} Inverse Problems 27 (2011), no. 1, 015009, 11 pp

\bibitem{SunZhou2017} J. Sun and A. Zhou, Finite Element Methods for Eigenvalue Problems, CRC Press, Boca Raton, FL, 2017.

\bibitem{Taylor1986}
M.E. Taylor, Partial Differential Equations. II. Qualitative Studies of Linear Equations, Springer-Verlag, New York, 1996.

\bibitem{XiGongSun2024} Y. Xi, B. Gong and J. Sun, {\em Analysis of a finite element DtN method for scattering resonances of sound-hard obstacles.} arXiv:2404.09300, 2024.
\end{thebibliography}
\end{document}